\documentclass[12pt]{amsart}
\usepackage{epsfig,color}

\makeatother

\theoremstyle{definition}

\def\fnum{equation}
\newtheorem{Thm}[\fnum]{Theorem}
\newtheorem{Cor}[\fnum]{Corollary}

\newtheorem{Lem}[\fnum]{Lemma}

\newtheorem{Pro}[\fnum]{Proposition}

\numberwithin{equation}{section}
\newcommand{\Vol}{{\text{Vol}}}

\newcommand{\nn}{{\bf{n}}}

\newcommand{\dist}{{\text {dist}}}

\newcommand{\PG}{\Psi_{\Gamma}}

\newcommand{\Hess}{{\text {Hess}}}

\def\RR{{\bold R}}

\def\SS{{\bold S}}

\newcommand{\dv}{{\text {div}}}
\newcommand{\e}{{\text {e}}}

\newcommand{\bF}{{\bar{F}}}

\newcommand{\cT}{{\mathcal{T}}}

\newcommand{\cG}{{\mathcal{G}}}
\newcommand{\cg}{{\mathfrak{g}}}

\newcommand{\cL}{{\mathcal{L}}}

\newcommand{\cO}{{\mathcal{O}}}

\newcommand{\cM}{{\mathcal{M}}}

\newcommand{\cQ}{{\mathcal{Q}}}

\newcommand{\cR}{{\mathcal{R}}}

\newcommand{\eqr}[1]{(\ref{#1})}

\title[Dynamics of closed singularities]{Dynamics of closed singularities}
\author[]{Tobias Holck Colding}%
\address{MIT, Dept. of Math.\\
77 Massachusetts Avenue, Cambridge, MA 02139-4307.}
\author[]{William P. Minicozzi II}%
 
\thanks{The  authors
were partially supported by NSF Grants DMS 1812142 and DMS 1707270.}

\email{colding@math.mit.edu  and minicozz@math.mit.edu}

\begin{document}

\maketitle

\begin{abstract}
Parabolic geometric flows have the property of smoothing for short time however, over long time, singularities are typically unavoidable,  can be very nasty and may be impossible to classify.  The idea of this paper is that, by bringing in the dynamical properties of the flow, we obtain also smoothing for long time for generic initial conditions.  When combined with our earlier paper, \cite{CM1}, this allows us to show that, in an important special case, the singularities are the simplest possible.

We take here the first steps towards understanding the dynamics of  the flow.  The question of the dynamics  of a singularity   has two parts.  One is:   What are the dynamics near a singularity?  The second is:  What is the long time behavior of the flow of things close to the singularity.  That is, if the flow leaves a neighborhood of a singularity, is it possible for the flow to re-enter the same neighborhood at a much later time?  The first part is addressed in this paper, while the second will be addressed elsewhere.
  
\end{abstract}

\section{Introduction}

The mean curvature flow, or MCF for short, is the negative gradient flow of volume on the space of closed hypersurfaces in Euclidean space.   
Under the mean curvature flow, a hypersurface locally moves in the direction where the volume element decreases the fastest.    The flow has the effect of contracting a closed hypersurface, eventually leading to its extinction in finite time.  The key to understand MCF is to understand the singularities that the flow goes through before it becomes extinct.  Singularities are modeled by their blow-ups, which are called tangent flows and are   shrinkers \cite{H2}, \cite{I1}.  A one parameter family of hypersurfaces $M_t$ flowing by the MCF is said to be a  shrinker (or self-similar around the origin in space-time) if they evolve by rescaling, that is, if $M_t=\sqrt{-t}\,M_{-1}$.  Round spheres and cylinders   evolve self-similarly under the mean curvature flow.

Suppose that $M_t$ is a one-parameter family of closed hypersurfaces flowing by MCF.  We would like to analyze the flow near a singularity in space-time.  After translating, we may assume that the singularity occurs at the origin in space-time.  If we reparametrize and rescale the flow as follows $t\to M_{-\e^{-t}}/\sqrt{\e^{-t}}$, then we get a solution to the rescaled MCF equation.    The rescaled MCF is the negative gradient flow for the $F$-functional (or Gaussian surface area)
\begin{align}
	F(\Sigma) = 	\int_{\Sigma} \e^{ - \frac{|x|^2}{4} } \, .
\end{align}
 The fixed points of the rescaled MCF, or equivalently the critical points of the $F$-functional, are the shrinkers.   The rescaling that takes place to get to the rescaled MCF has the effect of turning the question of the dynamics of the MCF near a singularity into a question of the dynamics near a fixed point for the rescaled flow.  

It follows from this   that we can treat the rescaled MCF as a special kind of dynamical system that is the gradient flow of a globally defined function and where the fixed points are the singularity models for the original flow.  

\vskip4mm
A one-parameter family of hypersurfaces  $M_t \subset \RR^{n+1}$  flows by mean curvature if
\begin{equation}
  x_t =  \bar{H} \, .
\end{equation}
Here $x$ is the position vector, 
  $\bar{H} = - H \, \nn$ is the mean curvature vector, $\nn$ is the outward unit normal,   and the mean curvature $H$ is given by
\begin{equation}	\label{e:defH}
	H = \dv \,  \nn = \sum_{i=1}^n \, \langle \nabla_{e_i} \nn , e_i \rangle \, .
\end{equation}
The $e_i$'s form an orthonormal frame for the hypersurface\footnote{With this convention, $H$ is $n/R$ on the $n$-sphere of radius $R$ in $\RR^{n+1}$ and $H$ is $k/R$ on the ``cylinder'' $\SS^k \times \RR^{n-k} \subset \RR^{n+1}$ of radius $R$.}.  The rescaled MCF is the equation 
\begin{align}
	x_t = \left( \frac{ \langle x , \nn \rangle}{2} - H \right) \, \nn \, .
\end{align}
The first variation formulas for volume and weighted volume show that the negative gradient flows for volume and the $F$ functional are MCF and rescaled MCF, respectively.
As mentioned, the fixed points for the rescaled MCF, or equivalently the critical points for
the $F$ functional, are  shrinkers that are self-similar around the origin.  The  shrinker equation is 
\begin{equation}
	H = \frac{\langle x , \nn \rangle}{2} \, .
\end{equation}

  In \cite{CM1}, we showed that the only smooth stable  shrinkers are spheres, planes, and generalized cylinders (i.e., $\SS^k \times \RR^{n-k}$).    In particular, the round sphere is the only closed stable singularity for the mean curvature flow.   A closed  shrinker is said to be stable if, modulo translations and dilations, the second derivative of the $F$-functional is non-negative for all variations at the given  shrinker, see \cite{CM1} for the precise definition as well as the definition of stability for non-compact  shrinkers.   
  
  We will here analyze the behavior of the rescaled flow in a neighborhood of a closed unstable shrinker.  We show that, in a suitable sense, ``nearly every'' hypersurface in a neighborhood of the unstable   shrinkers leaves the neighborhood, even after accounting for translations and dilations.  In contrast, in a small neighborhood of the round sphere, all closed hypersurfaces are convex and thus all become extinct in a round sphere under the MCF by a result of Huisken, \cite{H1}.  The point in space-time where a closed hypersurface nearby the round sphere becomes extinct may be different from that of the given round sphere.  This corresponds to that, under the rescaled MCF, it may leave a neighborhood of the round sphere but does so near a translation of the round sphere.   Similarly, in a neighborhood of an unstable  shrinker, there are closed hypersurfaces that under the rescaled MCF leave the neighborhood of the  shrinker but do so in a trivial way, namely, near a translate of the given unstable  shrinker.   However,  we will show that a typical closed hypersurface near an unstable shrinker not only leaves a neighborhood of the  shrinker, but, when it does, is not close to a rigid motion or dilation of the given  shrinker.

 Angenent, \cite{A}, constructed by ODE methods a  shrinking donut in $\RR^3$ together with similar higher dimensional examples. Ê Angenent's example was given by rotating a simple closed curve in the plane around an axis and, thus, had the topology of a torus. ÊIn fact, numerical evidence (see Chopp, \cite{Ch}, and Ilmanen, \cite{I3})
   suggests that, unlike for the case of curves, Êa complete classification of  shrinkers 
is impossible in higher dimensions as the examples appear to be so plentiful and varied; cf. also \cite{KKM}, \cite{K} and \cite{Nu}.  See  the surveys \cite{CM2} and \cite{CMP} for further discussion.
  
   \subsection{Dynamics near a closed  shrinker}

Let $E$ be the Banach space of $C^{2,\alpha}$ functions on a smooth closed embedded hypersurface $\Sigma \subset \RR^{n+1}$ with unit normal $\nn$.  We are identifying $E$ with the space of $C^{2,\alpha}$ hypersurfaces near $\Sigma$ by mapping a function $u$ to its graph 
\begin{equation}
	\Sigma_u = \{ p + u(p) \, \nn (p) \, | \, p \in \Sigma \} \, .
\end{equation}
If $E_1 , E_2$ are subspaces of $E$ with $E_1 \cap E_2 = \{ 0 \}$ and that together span $E$, i.e., so that 
\begin{equation}
	E = \{ x_1 + x_2 \, | \, x_1 \in E_1 , x_2 \in E_2 \} \, ,
\end{equation}
then we will say that $E= E_1 \oplus E_2$ is a splitting of $E$.

\vskip2mm
The next theorem is the main result about the dynamics near a  shrinker.  The essence of this theorem is that ``nearly every'' hypersurface in a neighborhood of the given unstable singularity leaves the neighborhood under the recaled MCF and, when it does, is not near a translate, rotation or dilation of the given singularity.  

\begin{Thm} 	\label{t:one}
Suppose that $\Sigma^n \subset \RR^{n+1}$ is a smooth closed embedded  shrinker, but is not a sphere.  There exists an open  neighborhood $\cO = \cO_{\Sigma}$ of $\Sigma$ and  a  subset $ W$ of $\cO$ so that:
\begin{itemize}
\item There is a splitting $E= E_1 \oplus E_2$  with  $\dim (E_1) > 0$ 
so that  $W$ is contained in the graph $(x,u(x))$ of a continuous mapping     $u: E_2 \to E_1$.
\item  If $v \in \cO \setminus W$, then the rescaled mean curvature flow starting at $v$ leaves $\cO$ and the orbit of $\cO$ under the group of conformal linear transformations of $\RR^{n+1}.  $\footnote{The group of conformal linear transformations of $\RR^{n+1}$ is generated by the rigid motions and the dilations.}
\end{itemize}
\end{Thm}

The space $E_2$ is loosely speaking the span of all the contracting directions for the flow together with all the directions tangent to the action of the conformal linear group.    It turns out that all the directions tangent to the group action are expanding directions for the flow.

Recall that the (local) stable manifold is the set of points $x$ near the fixed point so that the flow starting from $x$ is defined for all time, remains near the fixed point, and converges to the fixed point as $t\to \infty$.
Obviously, Theorem \ref{t:one} implies that the local stable manifold  is contained in $W$.  

There are several earlier results that analyze rescaled MCF  near a closed   shrinker, but all of these are for round circles and spheres which are stable under the flow.  The earliest are the global results of Gage-Hamilton, \cite{GH}, and Huisken, \cite{H1}, showing that closed embedded convex hypersurfaces flow to spheres.   There is a   later   estimate  of Sesum, \cite{Se},  on the rate of convergence  in Huisken's theorem.
 There is also the stable manifold theorem of Epstein-Weinstein, \cite{EW}, from the late 1980s for the curve shortening flow that also applies to closed immersed  shrinking curves, but does not incorporate the group action.     In particular, for something to be in Epstein-Weinstein's stable manifold, then under the rescaled flow it has to limit into the given  shrinking curve.  In other words, for a curve to be in their stable manifold it is not enough that it limit into a rotation, translation or dilation of the  shrinking curve.  
 Theorem \ref{t:one} deals with unstable critical points, where we do not have the geometric estimates of the convex case. The dynamics is greatly complicated by the action of the non-compact group of conformal linear transformations.

\subsection{The heuristics of the local dynamics}

We close this introduction by indicating why Theorem  \ref{t:one}   should hold.  Before getting to this, it is useful to recall the simple case of gradient flows near a critical point on a finite dimensional manifold.  Suppose therefore that $\bF: \RR^2 \to \RR$ is a smooth function with a non-degenerate critical point at $0$ (so $\nabla \bF(0) = 0$, but the Hessian of $\bF$ at $0$ has rank $2$).  The behavior of the negative gradient flow
\begin{equation}
(x' , y') = - \nabla \bF (x,y) 
\end{equation}
is determined by the Hessian of $\bF$ at $0$.  For instance, if $\bF(x,y)=\frac{a}{2}\,x^2+\frac{b}{2}\,y^2$ for constants $a$ and $b$, then the negative gradient flow solves the ODE's $x'=-a\,x$ and $y'=-b\,y$.  Hence, the flow lines are given by $x=\e^{-at}\,x(0)$ and $y=\e^{-bt}\,y(0)$.   
From this we see that the behavior  of the flow near a critical point depends on the index of the critical point.  The critical point $0$ is ``generic'', or dynamically stable, if and only if it has index $0$.  When the index is positive, the critical point is not generic and a ``random'' flow line will miss the critical point.

 We will next very briefly explain the underlying reason for the above theorem about the local dynamics near a closed  shrinker and why it is an infinite dimensional and nonlinear version of the simple finite dimensional example just discussed.  
 
 Suppose $\Sigma$ is a manifold and $f$ is a function on $\Sigma$.  
Let $w_i$ be an orthonormal basis of the Hilbert space $L^2(\Sigma,\e^{-f}\,d\Vol)$, where the inner product is given by $\langle v,w\rangle=\int_{\Sigma} v\,w\,\e^{-f}\,d\Vol$.  For constants $\mu_i\in \RR$ define a function $\bF$ on the infinite dimensional space $L^2(\Sigma,\e^{-f}\,d\Vol)$ as follows:  If $w\in L^2(\Sigma,\e^{-f}\,d\Vol)$, then
\begin{equation}
	\bF(w)=\sum_i\frac{\mu_i}{2}\,\langle w,w_i\rangle^2\, .
\end{equation}
As in the finite dimensional case, the negative gradient flow of $\bF$ is:
\begin{equation}
\Psi_t(w)=\e^{-\mu_i t}\langle w_i,w\rangle\, .
\end{equation}

Of particular interest is when $\Sigma^n\subset \RR^{n+1}$ is a  shrinker, $f(x)= \frac{|x|^2}{4}$, and the basis $w_i$ are eigenfunctions with eigenvalues $\mu_i$ of a self-adjoint operator $L$ of the form
\begin{equation}
L\, w=\cL\,w+|A|^2\,w+\frac{1}{2}w\, ,
\end{equation}
where $\cL w= \Delta w -\frac{1}{2}\langle x,\nabla w\rangle$ is the drift Laplacian.
The reason this is of particular interest is because in \cite{CM1} it was shown that the Hessian of the $F$-functional is given by
\begin{equation}
\Hess_F(v,w)=-\int_{\Sigma} v\,L\,w\,\e^{-\frac{|x|^2}{4}}\, .
\end{equation}

For an $\bF$ of this form,  the negative gradient flow is equal to the heat flow of the linear heat operator $(\partial_t-L)$.  Moreover, this linear heat flow is the linearization of the rescaled MCF at the  shrinker.   It follows that the rescaled MCF near the  shrinker is approximated by the negative gradient flow of $\bF$.  This  is also reflected by fact that if we formally write down the first three terms in the Taylor expansion of $F$, then we get the value of $F$ at $\Sigma$ plus a first order polynomial which is zero since $\Sigma$ is a critical point of $F$ plus a polynomial of degree two which is given by the Hessian of $F$ and is exactly $\bF$.   This gives a heuristic explanation for the above theorem:  The dynamics of the negative gradient flow of the $F$ functional should be well approximated by the dynamics for its second order Taylor polynomial.

 \section{Dynamics at an unstable critical point}	\label{s:one}
 
 In this section, we will prove a variation on the stable manifold theorem for dynamical systems in a neighborhood of a fixed point.  This will be applied later to the rescaled MCF near a  shrinker.  In this section, we will keep things   general,  assuming a few basic properties and making no reference to MCF.

\vskip2mm
Throughout this section, $E$ is a Banach space, 
$	\Psi: E \to E$ is a continuous map with $\Psi (0)=0$, 
$
 	T: E \to E
$
 is a bounded linear map, and $Q$ is a bounded positive definite symmetric bilinear form{\footnote{When we apply this, $E$ will be the Banach space of $C^{2,\alpha}$ functions on a  shrinker and the bilinear form $Q$ will be a weighted $W^{1,2}$ norm.}} on $E$. We will use $\|x \|$ and $\|x \|_Q = \sqrt{Q(x,x)}$ to denote the $E$-norm and $Q$-norm, respectively, of $x \in E$.  Since $Q$ is bounded, there is a constant $C_Q$ so that $|Q(x,y)| \leq C_Q \,  \|x \| \, \|y\|$ 
 for all $x, y \in E$.  In particular, 
 $\|x\|^2_Q \leq C_Q \, \|x\|^2$.

\vskip2mm
We will assume that $E,\Psi, T, Q$ satisfy the following conditions:
\begin{enumerate}
\item  There is a  splitting $E=E_1  \oplus E_2$ so that:
\begin{itemize}
\item $E_1$ and $E_2$ are   $Q$-orthogonal, i.e., $Q: E_1 \times E_2 \to 0$.
\item $E_1$ and $E_2$ are  $T$-invariant, i.e., $T: E_j \to E_j$  for $j=1,2$.
\item The $Q$-orthogonal projection $P_1: E \to E_1$ is continuous.
  \end{itemize}
 \item $T$ is $Q$-continuous and there exist $\lambda > 1$ and $\mu \in (0,\lambda)$ so that:
 \begin{itemize}
\item If $x \in E_1$, then $\|T(x)\|_Q \geq \lambda \, \|x\|_Q$.
\item 	 If $x \in E_2$, then $\|T(x)\|_Q \leq \mu \, \|x\|_Q$.
\end{itemize}
\item   Given $\epsilon > 0$, there exists $r>0$ so that if $x,y \in B_r \subset E$, then
\begin{equation}
	\left\| (\Psi - T) (x) - (\Psi - T) (y) \right\|_Q \leq \epsilon \, \|x-y\|_Q \, . \notag
\end{equation}   
\end{enumerate}

\vskip2mm
{\bf{Remarks}}:
\begin{itemize}
\item Property (2) says that $T$ is strictly expanding on $E_1$  in the $Q$-norm 
and is less expanding on $E_2$; often, $T$ will actually be contracting on $E_2$.
\item Property (3) is a  local  Lipschitz bound on $(\Psi -T)$ with respect to the $Q$-norm; we will refer to this as $Q$-Lipschitz.  Essentially,
 $T$ is the linear part of the Taylor expansion of $\Psi$ at the fixed point $0$.
\item
The $Q$-Lipschitz approximation (3) is only valid on a ball in the Banach space norm.  If this could be replaced by the $Q$-norm, then we would work just with the $Q$-norm.
\end{itemize}

The next lemma shows that the assumption in (1) that $P_1 : E \to E_1$ is continuous is always satisfied when $E_1$ is finite dimensional.

\begin{Lem}	\label{e:fd}
If $\dim (E_1) < \infty$, then $P_1$ is  continuous and, thus, so is $P_2 (x) = x - P_1 (x)$.
\end{Lem}

The proof is a standard consequence of  the following simple fact:

\begin{Lem}	\label{l:simple}
Let $E$ be a Banach space and $Q$ Ê a positive definite symmetric bilinear form on $E$. ÊIf $E_1$ is a subspace of $E$ and $\dim (E_1) < \infty$, then there exists $\kappa > 0$ so that
\begin{equation}
	 \kappa \, \|x\| \leq \|x\|_Q {\text{ for all }} x \in E_1 \, .
\end{equation}
\end{Lem}

\begin{proof}
Set $n = \dim (E_1)$ and let $v_1,\cdots,v_n$ be a $Q$-orthonormal basis for $E_1$ and set $\Lambda = \max_i \, \|v_i \|$. ÊIf $x = \sum_{i=1}^n x_i \, v_i $, then
\begin{equation}
	\|x\| \leq \sum_{i=1}^n \left\| x_i \, v_i \right\| \leq \Lambda \, \sum_{i=1}^n \left\| x_i Ê \right\| 
	\leq \Lambda \, \sqrt{n} \, \left( \sum_{i=1}^n Ê x_i^2 \right)^{ \frac{1}{2} } 
	= \Lambda \, \sqrt{n} \,  \|x\|_Q \, .
\end{equation}
 \end{proof}

\begin{proof}[Proof of Lemma \ref{e:fd}]
Since $P_1$ is linear, we must show it is bounded.
Given $x \in E$, we have that $P_1(x) \in E_1$ so Lemma \ref{l:simple} gives
\begin{equation}
	\kappa \, \left\| P_1(x)   \right\| \leq \left\| P_1(x)   \right\|_Q \leq \left\| x   \right\|_Q \leq \sqrt{C_Q} \, 
	\left\| x   \right\| \, , 
\end{equation}
where the second inequality used that $P_1$ is $Q$-orthogonal projection and the last inequality used that $Q$ is bounded.  
\end{proof}

It is convenient to let  
  $(x_1 , x_2) \in E_1 \oplus E_2$ be the coordinates of a point $x \in E$. 
  Let $\Psi_j$ denote $\Psi$ followed by  the $Q$-orthogonal projection $P_j$ to $E_j$.    
We will assume that $\epsilon > 0$ is small enough that
\begin{equation}	\label{e:mylambdas}
	\lambda -2\epsilon > 1 {\text{ and }}  \frac{\lambda -2\epsilon}{\mu+2\epsilon}  > 1 \, .
\end{equation}

Let $W$ be the set of points whose trajectories never leave the closed ball $\overline{B_r}$ 
\begin{equation}
	W = \{ x \in \overline{B_r} \, | \, \Psi^n (x) \in \overline{B_r} {\text{ for all }} n > 0 \} \, .
\end{equation}
Since $\Psi$ is continuous, $W$ is closed.
The next proposition shows that $W$ is  a graph over $E_2$.

\begin{Pro}	\label{l:w2}
$W$ is  the graph of a $Q$-Lipschitz mapping $u: P_2(W) \subset E_2 \to E_1$.  If $E_1$ is finite dimensional, then $u$ is Lipschitz.
\end{Pro}

The idea is that if $W$ was not a graph over $E_2$, then it would contain a pair of points whose difference 
was in the expanding direction for $T$.  Since $T$ closely approximates $\Psi$, repeatedly applying $\Psi$ will 
eventually take at least one of the points out of the ball.  This argument  gives a cone condition for $W$ that implies Lipschitz regularity.

\vskip2mm
 The proof is modeled on results from \cite{HiP}   for hyperbolic diffeomorphisms.  We start with a  
 lemma that shows that a cone condition is preserved when we apply $\Psi$.

\begin{Lem}	\label{l:basic}
If $  \epsilon> 0$, $r= r(\epsilon)$ is from (3), 
$x,y \in B_r \subset E$ and	$\|x_1 - y_1\|_Q \geq \|x_2 - y_2\|_Q$, then
\begin{align}	 
	 \left\| \Psi_1 (x) - \Psi_1 (y) \right\|_Q &\geq (\lambda -2 \, \epsilon) \, \|x_1 - y_1\|_Q \geq  
	 \left\| \Psi_2 (x) - \Psi_2 (y) \right\|_Q
	    \, .   \label{e:basic1a}
\end{align}
\end{Lem}

\begin{proof}
Set $T_j = P_j \circ T$.   
Since $\Psi_1 = P_1 \circ \left( T + (\Psi - T) \right)$ and   $T$ and $P_1$ are linear, we have
\begin{align}	\label{e:psi1T}
	\Psi_1 (x) - \Psi_1 (y) = T_1 (x-y) + P_1 
	\left(  (\Psi - T) (x) - (\Psi - T) (y) \right) \, .
\end{align}
Since  $P_1$ does not increase the $Q$-norm, (3) gives that
\begin{align}
	\left\| P_1 
	\left(  (\Psi - T) (x) - (\Psi - T) (y) \right) \right\|_Q &\leq 
	\left\|  
	   (\Psi - T) (x) - (\Psi - T) (y)   \right\|_Q  \leq \epsilon \, \|x-y\|_Q
	    \, .
\end{align}
Using this in \eqr{e:psi1T} and using that
  $T_1$ is uniformly expanding gives
\begin{equation}
	\left\|  \Psi_1 (x) - \Psi_1 (y) \right\|_Q  \geq \lambda \, \|x_1 - y_1\|_Q - \epsilon \, \|x - y\|_Q \geq
	 (\lambda - 2 \, \epsilon) \, \|x_1 - y_1\|_Q \, ,
\end{equation}
where the last inequality used  that $\|x_1 - y_1\|_Q \geq \|x_2 - y_2 \|_Q$.  This gives the first inequality in  \eqr{e:basic1a}. 
To get the second inequality in \eqr{e:basic1a},  observe that
\begin{align}	\label{e:basic2}
	\left\|  \Psi_2 (x) - \Psi_2 (y) \right\|_Q &\leq \left| T_2 (x-y) \right\|_Q + \left\| (\Psi - T) (x) -
	(\Psi - T) (y) \right\|_Q \notag  \\
	&\leq  \mu \,  \|x_2 - y_2\|_Q + \epsilon \, \|x-y\|_Q \leq (\mu  + 2 \epsilon) \, \|x_1 - y_1\|_Q \leq
	(\lambda - 2\epsilon) \, \|x_1 - y_1\|_Q  \, ,
\end{align}
where the last inequality used \eqr{e:mylambdas}.
\end{proof}

\begin{proof}[Proof of Proposition \ref{l:w2}.]
Suppose that $x , y \in W$.  We claim that
\begin{equation}	\label{e:Qlips}
	\|x_2 - y_2\|_Q \geq \|x_1 - y_1\|_Q \, .
\end{equation}
We will prove \eqr{e:Qlips} by contradiction. If \eqr{e:Qlips} fails, then 
Lemma \ref{l:basic} gives
\begin{equation}
	  \left\| \Psi_1 (x) - \Psi_1 (y)    \right\|_Q  \geq(\lambda -2 \, \epsilon) \,
	  \|x_1 - y_1\|_Q \geq   	 \left\| \Psi_2 (x) - \Psi_2 (y) \right\|_Q \, .
\end{equation}
Note that this implies that
  Lemma \ref{l:basic} also applies to $\Psi(x)$ and $\Psi (y)$ (these remain in $B_r$ by the definition of $W$), so that we can repeatedly apply the lemma to get
\begin{equation}
	2r \, C_Q \geq C_Q \, \left\| \Psi^n (x) - \Psi^n(y) \right\| 
	\geq \left\| P_1 (\Psi^n  (x) - \Psi^n (y)  )  \right\|_Q \geq \left( \lambda -2 \, \epsilon \right)^n \, \|x_1 - y_1\|_Q \, .
\end{equation}
Since $r$ is fixed and  $\lambda -2 \, \epsilon$ is strictly greater than one, this gives a contradiction when  $n$ is sufficiently large.
Therefore, we conclude that \eqr{e:Qlips}  holds as claimed.

The first consequence of \eqr{e:Qlips} is that $W$ is a graph over $E_2$.  Namely, if $x,y \in W$ and $x_2 = y_2$, then
\eqr{e:Qlips} implies that
$x_1 = y_1$.  Define the subset $W_2 \subset E_2$ by
\begin{equation}
	W_2 = \{ P_2 (x) \, | \,   x \in W \} \, .
\end{equation}
Define a map $u: W_2 \to E_1$ by $u(x_2) = x_1$ where $(x_1 , x_2) \in W$.  It follows from \eqr{e:Qlips} that
\begin{equation}
	\left\| u(x_2) - u(y_2) \right\|_Q \leq \|x_2-y_2\|_Q 
\end{equation}
for every $x_2 , y_2 \in W$.  In other words, the mapping $u$ is $Q$-Lipschitz with norm one.

Finally, suppose that $E_1$ is finite dimensional.  Lemma \ref{l:simple} gives  $\kappa > 0$ so that if $z \in E_1$, then
$
	\kappa \, \|z\| \leq \|z \|_Q \, .
$
Therefore, if $x , y \in E_2$, then
\begin{equation}
	\kappa \, \|u(x) - u(y)\| \leq \|u(x) - u(y)\|_Q \leq |x-y|_Q \leq \sqrt{C_Q} \, \|x-y\| \, ,
\end{equation}
where the second inequality used that $u$ is $Q$-Lipschitz and the last inequality used that $Q$ is bounded.  It follows that $u$ is Lipschitz.
\end{proof}

\subsection{A group action}	\label{ss:groupO}

We will now extend the results from the previous subsection to allow for an action by a group $\cR$ on $E$.  Let $\cR_0$ be the orbit of $0$ under the $\cR$ action.
We will assume that $\cR$ has the the following properties:
\begin{enumerate}
\item[($\cR 0$)] $\cR$ commutes with $\Psi$ and is $2$-bi-Lipschitz on a neighborhood of $0$ in $E$: If $g \in \cR$ and $|x|, |y|, |g(x)|, |g(y)| < \bar{r}$ for some $\bar{r} > 0$, then 
\begin{align}
	\frac{1}{2} \, \|x-y\| \leq \|g(x) - g(y)\| \leq 2 \, \|x-y\| \, .
\end{align}
\item[($\cR 1$)] $E_1$ is transverse to   $\cR_0$:  There exists $r_0 > 0$ and a continuous strictly increasing function on $d_0:[0,r_0)$ with $d_0 (0) =0$, 
so that if $|x| < r_0$ and $\|x_2\|_Q \leq \|x_1\|_Q$, then
\begin{align}
	\dist_{E} (x , \cR_0) \geq d_0 ( \|x\|) \, .
\end{align}
\item[($\cR 2$)] To first order, $\cR$ is non-contracting on $E_1$ and non-expanding on $E_2$:  There exist $r_1 > 0$ and a continuous function $\delta_0$ on $\RR$ with $\delta_0 (0) = 0$ so that if $r\leq r_1$,  $g \in \cR$, $\|x\| , \|y\| , \|g(x)\| < \frac{ r}{3}$, then
$\|g(y)\| < r$ and
\begin{align}
	\|x_1 - y_1\|_Q - \delta_0(r) \, \|x-y\|_Q &\leq  \|(g(y) - g(x))_1\|_Q \, , \\
	 \|(g(y) - g(x))_2\|_Q &\leq   \|x_2 - y_2\|_Q +   \delta_0(r) \, \|x-y\|_Q \, .
\end{align}
\end{enumerate}

 Let $s>0$ be a small constant to be chosen and let $W_0$ be the set of points whose trajectories never leave the (closed)  $s$-tubular neighborhood of the orbit $\cR_0$ under the action of $\Psi$
\begin{equation}
	W_0 = \{ x \in E \, | \, {\text{ for all $n\geq 0$ there exists  $g_n \in \cR$ so that }} g_n(\Psi^n (x)) \in \overline{B_s}  \} \, .
\end{equation}
Since $\Psi$ and the action are continuous, $W_0$ is closed.
The next proposition shows that $W_0$ is  a graph over $E_2$.

\begin{Pro}	\label{l:w2cR}
If $s> 0$ is sufficiently small, then $B_s \cap W_0$ is  the graph of a $Q$-Lipschitz mapping $u: P_2(W_0) \subset E_2 \to E_1$.  If, in addition, $E_1$ is finite dimensional, then $u$ is Lipschitz.
\end{Pro}

\begin{proof}
Suppose that $x \in B_s \cap W_0$ and let $y \in B_s$ be any point with
\begin{align}	\label{e:xystart}
	\|x_2 - y_2\|_Q < \|x_1 - y_1\|_Q \, .
\end{align}
 The first part of the proposition follows if we show that $y \notin W_0$.
 
 Define sequences of points $x^i$ and $y^i$ as follows:
 \begin{itemize}
 \item Set $x^1 = x$ and $y^1 = y$.
 \item For each $i>1$, choose $g_i \in \cR$ so that $\|g_i (\Psi (x^{i-1}))\| < s$ and then set $x^i = g_i (\Psi (x^{i-1}))$ and $y^i = g_i (\Psi (y^{i-1}))$.
 \end{itemize}
 Fix some small $r_1 > 0$ (to be chosen small and then choosing $s \in (0, d_0 (r_1)$).  Repeatedly applying Lemma \ref{l:basic} and ($\cR 2$), 
 it follows that there exists $\kappa > 1$ so that
  \begin{enumerate}
 \item[(C1)] If $\|y_i \| < r_1$ for all $i <n$, then 
 \begin{equation}
	  \left\| (x^n - y^n)_1    \right\|_Q  \geq \kappa^{n-1} \,
	  \|x_1 - y_1\|_Q \geq   	 \left|\| ( x^n - y^n)_2 \right\|_Q \, .
\end{equation}
 \end{enumerate}
Since $\kappa > 1$ and the $Q$-norm is continuous, there must be 
  a first $n$ so that $r_1 \leq \|y^n\|$.  Once we have this, then ($\cR 1$) implies that 
  \begin{align}
  	\dist_E (y^n , \cR_0) \geq d_0 (r_1) > 2\, s \, .
  \end{align}
  Let $g = g_n \circ g_{n-1} \circ \dots \circ g_2$, so that $y^n = g_n (\Psi^{n-1}(y))$ by the first part of ($\cR 0$).
 Since $\cR$ preserves the orbit $\cR_0$, it follows from this and the second part of ($\cR 0$) that
   \begin{align}
  	\dist_E (\Psi^{n-1}(y) , \cR_0)  = \dist_E ( g^{-1} (y^n) , \cR_0)  \geq \frac{1}{2} \, \dist_E (y^n , \cR_0) >   s \, .
  \end{align}
  In particular, $y$ is not in $W_0$, completing the proof of the first part.

\vskip2mm
Finally, the second  claim follows as in Proposition \ref{l:w2}.
\end{proof}

\section{The dynamics of rescaled MCF}

We will apply the dynamics results from the previous section to study rescaled MCF in a neighborhood of a smooth closed embedded  shrinker $\Sigma$ that is not a round sphere.    

\subsection{The Banach space $E$, the map $\Psi$, and the norm $Q$}

The Banach space $E$ will be the H\"older space of $C^{2,\alpha}$ functions on $\Sigma$, so the $\| \cdot \|_{E}$ is the $C^{2,\alpha}$ norm.
Define the map $\Psi_t$ to be the time $t$ rescaled MCF acting on the space $E$.  Since $\Sigma$ is a fixed point of  
the rescaled MCF,   $\Psi_t (0) = 0$ for all $t$.  Set $\Psi = \Psi_1$.

We will use the second variation operator $L$ of $\Sigma$ to define the norm $Q$, the splitting $E=E_1 \oplus E_2$, and the linear map $T:E \to E$.    Recall that, by \cite{CM1}, 
\begin{align}
	\cL &=  \Delta  - \frac{1}{2} \, \langle x , \nabla \cdot \rangle \, , \\
	L &=\cL + |A|^2 + \frac{1}{2}  \, .
\end{align}
These operators are symmetric with respect to the Gaussian $L^2$ norm 
\begin{align}
	(u,v) \to \int_{\Sigma} u v \, \e^{ - \frac{|x|^2}{4} }  \, .
\end{align}
Fix a   positive constant $\Lambda > |A|^2 + \frac{3}{2}$ and define
 the bilinear form $Q$ by
\begin{equation}
	Q(u,v) = \int_{\Sigma} \left\{ \Lambda \, u \, v  -  u \, L v   \right\} \, \e^{ - \frac{|x|^2}{4} } \, .
\end{equation}
Since the weight $\e^{ - \frac{|x|^2}{4} }$ is bounded by one, $Q$ is bounded by a constant times $E$.   
Since $L$ is symmetric with the Gaussian $L^2$ norm, it follows 
that $Q$ is symmetric and, moreover, that $L$ is also symmetric with respect to   $Q$.  
Finally, observe that $Q$ is bounded above and below by the Gaussian $W^{1,2}$ norm. 
	
\subsection{The splitting and the map $T$}
	By corollary $5.15$ in \cite{CM1} (with the obvious modifications{\footnote{\cite{CM1} used the Gaussian $L^2$ norm; the extension to the $Q$ norm follows with obvious modifications.}}), theorem $5.2$ in \cite{CM1}, and theorem $4.30$ in \cite{CM1}, we have:
\begin{itemize}
\item  $L$ has a complete $Q$-orthonormal basis of eigenfunctions $w_i$ with $L \, w_i = - \mu_i \, w_i$, where the eigenvalues $\mu_i$ go to infinity.
\item  $L \, H = H$.  If $v$ is a constant vector field and $\nn$ is the unit normal, then $L \, \langle v , \nn \rangle = \frac{1}{2} \, \langle v , \nn \rangle$.  
If  $z$ is a vector field generating a rotation, then $L \, \langle z , \nn \rangle = 0$.{\footnote{A translation of $\RR^{n+1}$ is generated by a  vector   $v$ and a dilation
  is generated by the vector field $x$.  Taking the normal parts of these   gives the vector fields
$\langle v , \nn\rangle$ and $\langle x , \nn \rangle = 2 \, H$.}}
\item The lowest eigenvalue $\mu_1 < -1$ (since $\Sigma$ is not a round sphere).
\end{itemize}

Let $E_1$ be the span of the eigenspaces with eigenvalues less than $-1$, i.e.,
\begin{equation}
	E_1 =  {\text{ Span }} \{ w_i \, | \, \mu_i < -1 \} \, .
\end{equation}
Since the $\mu_i$'s go to $\infty$, we have $0 < \dim (E_1) < \infty$.  By Lemma \ref{e:fd}, $P_1$ is continuous.

Let $E_2$ be the span of the eigenspaces with eigenvalue at least $-1$, so the $Q$-orthogonality of the $w_i$'s implies that $E_1$ and $E_2$ are $Q$-orthogonal.  
The vector fields generating rotations, dilations and translations are all contained in $E_2$; this will be important later.

The linear map $T:E \to E$ is defined on the $Q$-basis $\{ w_i \}$ by
\begin{equation}	\label{e:defT}
	T \, w_i = \e^{-\mu_i} \, w_i \, .
\end{equation}
It is clear  that $T$ preserves each $E_j$ and is $Q$-bounded (since the $\mu_i$'s are bounded from below).  Property (2) also follows immediately with $\mu  = \e$ and
$\lambda = \e^{-\mu_j}$ where $\mu_j$ is the largest eigenvalue below $-1$.  To see that $T$ is bounded, observe that $T$ can alternatively be defined by $T (w)(x) = w(x,1)$ where $w(x,t)$ is the solution of the linear parabolic equation
\begin{align}
	\partial_t \, w (x,t) &= L \, w (x,t) \, ,  
\end{align}
with initial condition
		$w(x,0)  = w(x)$.
		Interior Schauder estimates for linear equations (e.g., theorem $4.9$ in \cite{L2}) then implies that
$\| w(\cdot , 1) \|_{C^{2,\alpha}} \leq C \, \|w \|_{C^0}$ and, thus,  $\|T(w)\| \leq C \, \|w \|$.

The above defines all of the objects needed for the dynamical system and we have verified all of the needed properties except for three:
\begin{itemize}
\item   $\Psi$ is defined on a neighborhood of $0$.
\item  $\Psi$ is continuous.
\item $\Psi$ satisfies the $Q$-Lipschitz approximation property (3).
\end{itemize}
These  will be proven in the next two sections.

\section{Local existence for rescaled MCF: $\Psi$ is defined near $0$}

In this section, we will look at the rescaled MCF of graphs over a fixed  smooth closed embedded  shrinker $\Sigma$.      The next lemma establishes  local existence of the rescaled MCF $\Psi_t$ for $t \leq 1$ and shows that
    it is continuous at $0$; this is well-known to experts, but the exact dependence is needed here and does not appear to be recorded in the literature.    Analogous results for graphical mean curvature flow were proven by Lieberman, \cite{L1}, and Huisken, \cite{H3}, and the results in this section follow similarly.
    
\begin{Lem}	\label{t:localF}
There exists $\delta_1$, $\epsilon$ and $\alpha  > 0$ and $C$
so that if $w \in C^{2,\alpha}$ satisfies 
\begin{equation}
 |w| + |\nabla w| \leq \delta \leq \delta_1 {\text{ and }} |\Hess_w| \leq 1 \, , 
\end{equation}
 then there is a  solution of rescaled MCF 
$u: \Sigma \times [0,1] \to \RR$ with $u(x,0) = w(x)$ and
\begin{itemize}
\item $|u|  \leq  C \, \delta$,  $|\nabla u|  \leq  C \, \sqrt{\delta}$ and $|\Hess_u| \leq C$.
\item  $\| u (\cdot , 1 ) \|_{C^{2,\alpha}} \leq C \, \delta^{\epsilon }$.
\end{itemize}
\end{Lem}

Lemma \ref{t:localF} verifies the first of the three remaining properties  for the dynamical system.

\vskip2mm
We will first establish uniform bounds for the solutions and, in the process, prove 
Lemma \ref{t:localF}.  We will then show that $\Psi$ is continuous.
In the next section, we will use these bounds and the finer structure of the nonlinearity to establish the    $Q$-Lipschitz approximation property.

\subsection{The graph equation for rescaled MCF}

In this subsection, we give the basic properties of graphical rescaled mean curvature flow equation
$
	\partial_t \, u = \cM \, u $.
	The next lemma shows that the graphical rescaled MCF equation is quasilinear and uniformly parabolic so long as  $|\nabla u|$   and $|u|$ are sufficiently small.  

\begin{Lem}	\label{p:quasi}
\cite{CM3}
The equation $\partial_t \, u = \cM \, u $ is the quasilinear parabolic equation 
\begin{align}
	\partial_t \, u = \Omega (x, u , \nabla u)  + \Phi_{\alpha \beta} (x, u, \nabla u) \, u_{\alpha \beta} \, ,
\end{align}
where $\Omega (x,s,y)$ and $\Phi (x,s,y)$ depend smoothly on $x,s,y$ as long as $|s|$ and $|y|$ are sufficiently small
and $\Phi_{\alpha \beta} (x, 0, 0)  = \delta_{\alpha \beta}$ is the identity matrix.   
\end{Lem}

The next lemma writes the graphical rescaled MCF equation as a perturbation of the linearized equation.  The   nonlinearity $\cQ(u)$ is  essentially quadratic, so   $\cQ(u) - \cQ(v)$ is   bounded by $C_{u,v} \, (u-v)$ where   $C_{u,v}$ is small when $u$ and $v$ are.

\begin{Lem}	\label{t:nonlinQ}
We have $\cM \, u = L \, u + \cQ(u)$, where 
the nonlinearity $\cQ$ satisfies
\begin{enumerate}
\item[(Q)] There exist $C$ and $\epsilon > 0$ so that if   $\| u \|_{C^1} \leq \epsilon $ and  $\| v \|_{C^1} \leq \epsilon$, then  
\begin{align*}
	\cQ(u)  - \cQ(v)  = f   + \dv_{\Sigma} (W ) + \langle \nabla \bar{h}_u , V \rangle 
		+ \langle \nabla \left( h + H \, (u-v) \right), \bar{V}_v \rangle 	\, ,
\end{align*}
where $f$, $h$ and $\bar{h}_u$ are smooth functions, $H$ is the mean curvature of $\Sigma$,
 and  $V$, $\bar{V}_v$ and $W$  are smooth vector fields satisfying:
\begin{align}
	|f|,  |h|, | {W} |  &\leq C\, \left( \| u \|_{C^1} + \| v \|_{C^1} \right) \, \left( |u-v| + |\nabla u - \nabla v| \right) \, , \\
	   | V |  &\leq C\,  \left( |u-v| + |\nabla u - \nabla v| \right) \, , \\
	|\bar{h}_u|,   | \bar{V}_v |  &\leq C\, \left( \| u \|_{C^1} + \| v \|_{C^1} \right) \, , \\
 	\left| \nabla \bar{h}_u \right| &\leq C \, \| u \|_{C^1}  \, \left( 1 + \left| \Hess_u \right| \right) \, ,
	\\
	 |  \dv_{\Sigma} (\bar{V}_{v})| &\leq C\,  \| v \|_{C^2} \, .
	%\left| \nabla h \right| &\leq  C
	%\,
	%\left( 1 + \| v \|_{C^2} \right)  
	% \, \left(  |u-v| + \left| \nabla u - \nabla v \right| \right) 	 + C \, \| u \|_{C^1}  \, \left| \Hess_u - \Hess_v \right|  \, . 
	%\| f \|_{C^1} , \| h \|_{C^1} ,  \|  \bar{h}_u \|_{C^1} , & \| V \|_{C^1} ,  \|  \bar{V}_v \|_{C^1} , \|  {W} \|_{C^1} 
	% \leq C\, \left( \| u \|_{C^2} + \| v \|_{C^2} \right) \, .
\end{align}
\end{enumerate}
Finally, we have 
\begin{align}	\label{e:lastone35}
	 |\dv_{\Sigma} (W) | &\leq C\, \left( \| u \|_{C^1} + \| v \|_{C^1} \right) \, \left( |u-v| + |\nabla u - \nabla v| + |\Hess_u - \Hess_v| \right) \notag  \\
	 &\quad + C\, \left|\Hess_u \right| \, \left( |u-v| + |\nabla u - \nabla v| \right)\, ,   \\
	|\nabla h|   &\leq C(1 + |\Hess_u|) \,  \left( |u-v| + |\nabla u - \nabla v| \right)  + C \, (|u| + |\nabla u| ) \,  |\Hess_u - \Hess_v| \, . \label{e:lastone35b}
\end{align}
\end{Lem}

\vskip2mm
 Lemma \ref{t:nonlinQ} will be proven in an appendix.

\subsection{Local existence for the graph equation}

We are now prepared to prove Lemma \ref{t:localF}.  Using Lemma \ref{p:quasi}, we can write the equation as a quasilinear parabolic equation.   The argument follows the approach for graphical MCF in \cite{H3}, \cite{L1} with three steps:
\begin{itemize}
\item Bound $|u|$ and $|\nabla u|$ so that the equation becomes uniformly parabolic.
\item Use the $C^{\alpha}$ estimate on $\nabla u$ for uniformly parabolic quasilinear equations.
\item Appeal to the   Schauder estimates for linear equations.
\end{itemize}

Short time existence follows from standard arguments, but it also follows directly from short time existence for  MCF   together with the  relationship between MCF and rescaled MCF.  The point   is to obtain uniform estimates along the flow.
The first  step is  to establish uniform estimates for the height of the graph.  This is  done in the next  lemma  for a solution
$
 	u: \Sigma \times [0,\epsilon] \to \RR
$
  of the graphical rescaled MCF equation, where $\epsilon \in (0,1]$.

\begin{Lem}	\label{l:heightest}
There exist  $C$ and $\delta > 0$ so that if   $\sup_{\Sigma} |u(\cdot , 0)| \leq \delta$, then
\begin{equation}	\label{e:height}
	\sup_{\Sigma \times [0,\epsilon]} \,   |u| \leq C \, \sup_{\Sigma} \, |u(\cdot , 0)| \, .
\end{equation}
\end{Lem}

\begin{proof}
We first   bound  the positive part of the maximum of $u$.
Given  $t$, choose $p\in \Sigma$ with 
\begin{equation}
	u(p,t) = \max_{x} u(x,t) \, .
\end{equation}
We may assume that $u(p,t) > 0$ since we are otherwise done.
By the first derivative test,  $\nabla u(p,t) = 0$.  The second derivative test gives that $u_{\alpha \beta}(p,t)$ is negative semi-definite.  

By Lemma \ref{p:quasi},  $\Omega(p,0,0)=0$, $\Phi_{\alpha \beta}(p,0,0) = \delta_{\alpha \beta}$, and both $\Omega$ and $\Phi_{\alpha \beta}$ are smooth as long as $|u|$ is sufficiently small.  In particular, there exist $\delta_1 > 0$ and $C_1 > 0$ so that
if $s \leq \delta_1$, then
\begin{itemize}
\item $ \Phi_{\alpha \beta}(p,s,0)$ is positive definite and 
  $\left| \Omega (p,s,0)\right| \leq C_1 \, s$.
\end{itemize}
In particular, as long as the maximum $u(p,t)$ is at most $ \delta_1$, then we have that
\begin{equation}	\label{e:ttrt}
	\partial_t \, u (p,t)= \Omega (p, u(p,t) , 0)  + \Phi_{\alpha \beta} (p, u(p,t), 0) \, u_{\alpha \beta}(p,t) 
	\leq C_1 \, u(p,t) \, .
\end{equation}
 From this, it will follow  for $x \in \Sigma$ and $t \in [0,\epsilon]$  that
 \begin{equation}	\label{e:wantit}
 	u(x,t) \leq \e^{C_1 \, t} \, \max_{x} |u|(x,0)
\end{equation}
 as long as
$\sup_{x} u(x,0) \leq \delta \equiv \frac{1}{2} \, \e^{-2\,C_1} \, \delta_1$.

We will prove \eqr{e:wantit} by contradiction, so suppose that there exists $\bar{t} \in (0,\epsilon]$ so that \eqr{e:wantit} fails 
 at time $\bar{t}$.  In particular, we can choose   $\kappa>0$ (but less than $\min \{ C_1 , 1\} $) so that
\begin{equation}
	\e^{- (C_1 + \kappa) \, \bar{t}}  \, \max_{x} u(x,\bar{t}) > \kappa \, \delta_1 + \max_{x} u(x,0) \, .
\end{equation}
We will get a contradiction from this.
Define an auxiliary function 
\begin{equation}
	v (x,t) = \e^{- ( C_1+ \kappa) \, t} u(x,t)  \, .
\end{equation}
It follows that $v(x,0) \leq  \max_{x} |u(x,0)|$ and $\max_x v(x,\bar{t}) > \kappa \, \delta +  \max_{x} |u(x,0)|$. 
Let $T < \bar{t}$ be the smallest time that the maximum of $v$ on $\Sigma \times [0,T]$ is at least 
$\kappa \, \delta + \max_{x} |u|(x,0)$.  It follows that the maximum of $v$ on $\Sigma \times [0,T]$ occurs at a point $(p,T)$.
 Since this is the first time, we have $v(p,T) = \kappa \, \delta + \max_{x}| u|(x,0) \leq 2 \, \delta $ and, thus, 
 \begin{equation}	\label{e:gotttr}
 	u(p,T) = \e^{ (C_1 + \kappa) \,T} \,  v(p,T) \leq  2 \, \delta \,  \e^{ (C_1 + \kappa) \,T} \leq \delta_1 \, .
 \end{equation}
  By the first derivative test in time, we have
\begin{equation}
	0 \leq \e^{ (C_1 + \kappa) \,T} \,   \partial_t v (p,T) 
	=   \partial_t u (p,T) - (C_1 + \kappa) \, u (p,T) \, .
\end{equation} 
 However, \eqr{e:gotttr} allows us to apply \eqr{e:ttrt} at $(p,T)$, giving the desired contradiction.  Thus, we get that at each point 
 $x\in \Sigma$ and time $t\in [0,\epsilon]$ we have
 \begin{equation}
 	\max \{ 0 , u(x,t) \}  \leq \e^{C_1 \, t} \,  \max_{x}| u|(x,0) \, .
 \end{equation}
 The bound for the negative part follows by the same argument, but with the inequality on the $\partial_t$ derivative and on the Hessian of $u$ reversed.
 \end{proof}

 We will need the following standard maximum principle argument:
 
\begin{Lem}	\label{l:maxp}
If $f :  M_t \times [0 , T] \to [0 , \infty)$ satisfies $\max_{M_0} f   \leq C$ and
\begin{align}
	\left( \partial_t - \Delta_{M_t} \right) \, f \leq 2\,  f^2 \, ,
\end{align}
  then 
 $\max_{M_t}  f   \leq 2\, C$ for $t \leq \frac{1}{4C} $.
\end{Lem}

\begin{proof}
Define $m(t) = \max \{ w (x,s) \, | \, s \leq t \}$.  A standard argument shows that 
  $m(t)$ satisfies the differential equality $m'(t) \leq 2\, m^2 (t)$.  In particular, 
\begin{align}
	\left( \frac{1}{m(t)} \right)' = \frac{ - m'(t)}{m^2(t)} \geq - 2 \, .
\end{align}
Since $m(0) \leq C$, integrating this gives that $\frac{1}{m(t)} - \frac{1}{m(0)} \geq -2\,t$ and, thus, 
$
	m(t) \leq \frac{1}{\frac{1}{C} - 2 t} $.   It follows that $m(t) \leq 2 \, C$ as long as $t \leq\frac{1}{4C}$.
\end{proof}

We next apply this to get a short-time uniform curvature estimate for MCF.

\begin{Cor}	\label{e:cbounda}
If $M_t$ is a MCF with $\sup_{M_0} |A|^2 \leq C$, 
 then $\sup_{M_t} |A|^2 \leq 2 \, C$ for $t \leq\frac{1}{4C}$.
\end{Cor}

\begin{proof}
Simons' equation, theorem $3.2$ in \cite{HP},  for $|A|^2$ gives that
\begin{align}
	\left( \partial_t - \Delta_{M_t} \right) \, |A|^2 &= -2\,  |\nabla A|^2   + 2 |A|^4 \, . 
\end{align}
The Corollary now follows by applying Lemma \ref{l:maxp} with $f= |A|^2$.
\end{proof}

\begin{Pro}	\label{p:mainpro}
Given $\Sigma$, 
 there exists $\delta_0$,  $\alpha'$, $\epsilon_0 > 0$ and $C$
so that if $w: \Sigma \to \RR$ satisfies 
\begin{equation}
 |w| + |\nabla w| \leq \delta \leq \delta_0 {\text{ and }} |\Hess_w| \leq 1 \, , 
\end{equation}
 then there is a  solution of rescaled MCF 
$u: \Sigma \times [0,\epsilon_0] \to \RR$ with $u(x,0) = w(x)$ and
\begin{enumerate}
\item[(A)] $|u|  \leq  C \, \delta$,  $|\Hess_u| \leq C$, and  $|\nabla u|^2  \leq  C \, \delta$ on $ \Sigma \times [0,\epsilon_0] $.
\item[(B)] $\| u (\cdot , \epsilon_0 )\|_{C^{2,\alpha'}} \leq C$.
\item[(C)] Given $\alpha \in [0,\alpha')$, we have $\| u (\cdot , \epsilon_0 ) \|_{C^{2,\alpha}} \leq C \, \delta^{ \frac{\alpha' - \alpha}{2+ \alpha'} }$.
\end{enumerate}
\end{Pro}

\begin{proof}
The first bound in (A) follows from Lemma \ref{l:heightest}.  The second bound in (A) follows from Corollary \ref{e:cbounda}
and the relationship between MCF and the rescaled MCF; this is where $\epsilon_0 > 0$ is chosen.
The third bound in (A) follows from the first two bounds and the interpolation inequality 
\begin{equation}
	\sup_{\Sigma} |\nabla f|^2 \leq C \, \left( \sup_{\Sigma} |  f| \right) \, \left( \sup_{\Sigma} |\Hess_f| \right) \, .
\end{equation}
 
Using the uniform $|u|$ and $|\nabla u|$ bounds from (A), Lemma \ref{p:quasi} implies that the graphical mean curvature flow equation is uniformly parabolic and we get  an interior H\"older gradient estimate (see $12.10$ in \cite{L2}; cf. \cite{LSU}, $4.6$ in \cite{L1}):
 \begin{equation}
 	\| \nabla u \|_{C^{\alpha'}(\{ t \geq \frac{\epsilon_0}{4} \})} \leq C' \, , 
 \end{equation}
 for   constants $C'$ and $\alpha' > 0$ depending on the other bounds thus far.  Since the space-time gradient of $u$ is bounded, $u$ is also H\"older continuous.  Thus,    the chain rule gives
 \begin{align}
 	\| \Phi_{\alpha \beta} (x, u, \nabla u) \|_{C^{\alpha'}(\{ t \geq \frac{\epsilon_0}{4} \})} &\leq C'   \, , \\
	\| \Omega (x, u, \nabla u) \|_{C^{\alpha'}(\{ t \geq \frac{\epsilon_0}{4} \})} &\leq C'   \, .
 \end{align}
 We can now appeal to the interior Schauder estimates (theorem $4.9$ in \cite{L2})
 for the linear equation to get  an interior $C^{2,\alpha'}$ bound on $u$ 
 \begin{align}	\label{e:almostwww}
 	\| u \|_{C^{2,\alpha'}(\{ t \geq \frac{\epsilon_0}{3} \})} \leq C \, \left( \sup |u(\cdot , 0)| +   
	\| \Omega (x, u , \nabla u) \|_{C^{\alpha'}(\{ t \geq \frac{\epsilon_0}{4} \})} \right) 
	\leq C' \, .
 \end{align}
This gives (B).  
To get (C), we use the   interpolation inequality  (see page $141$ of \cite{GT})
\begin{equation}
	\| u \|_{C^{k,\alpha}(\Sigma)} \leq C \, \left(  \| u \|_{C^{k_1,\alpha_1}(\Sigma)} \right)^{\mu} \,
	\left( \| u \|_{C^{k_2,\alpha_2}(\Sigma)}  \right)^{1-\mu} \, , 
\end{equation}
where $0< \mu< 1$ and
\begin{equation}
	k+ \alpha = \mu (k_1 + \alpha_1) + (1-\mu ) \, (k_2 + \alpha_2 ) \, .
\end{equation}
If we set $k=k_1 = 2$, $k_2 = 0$, $\alpha_1 = \alpha'$  and $\alpha_2  = 0$, then 
$
	\mu = \frac{2+ \alpha}{2+ \alpha'} 
$
and we get
(C).
\end{proof}

Iterating this gives essentially the same corollary, but on the unit time interval $[0,1]$.

\begin{Cor}	\label{c:mainpro}
Choosing $\delta_0 > 0$   smaller and $C$ larger, Proposition \ref{p:mainpro} holds with $\epsilon_0 = 1$.
\end{Cor}

\begin{proof}
  Proposition \ref{p:mainpro} gives a solution $u$ for $t \leq \epsilon_0$.  However,  property (5) implies that the $C^2$ norm of $u(\cdot , \epsilon_0)$ is   small (i.e., bounded by a positive power of $\delta$).  Therefore, we can apply 
   Proposition \ref{p:mainpro}  again but this time with $w (x) = u(x , \epsilon_0)$.
  After iterating Proposition \ref{p:mainpro} approximately $1/\epsilon_0$ times, we get a solution up to $t=1$. 
  \end{proof}

 \vskip2mm
 
\begin{proof}[Proof of Lemma \ref{t:localF}]  
This follows immediately from Corollary \ref{c:mainpro}.
\end{proof}

\section{The $Q$-Lipschitz approximation property}

We will prove that the mapping $\Psi$ given by time one rescaled MCF is continuous:

 \begin{Cor}	\label{l:ctypsi}
There exist $\alpha > 0$ and a neighborhood of $0$ in $C^{2,\alpha}$ so that $\Psi$ is continuous.
 \end{Cor}

 Furthermore, we will  prove that $\Psi$ 
 has the $Q$-Lipschitz approximation property (3),   i.e., the $Q$-Lipschitz norm of $\Psi - T$ is small. 
 Thus, let $u$ be a  solution of the nonlinear  equation 
\begin{align}	\label{e:uQ}
	\left( \partial_t - L \right) \, u = \cQ (u) \, ,
\end{align}
where the nonlinear $\cQ$ satisfies property (Q) from Lemma \ref{t:nonlinQ} 
on a closed  shrinker $\Sigma$.

\begin{Pro}	\label{p:lips}
Given $C_2$, 
there exist
  $\delta_1 > 0$, $\epsilon > 0$ and $C_1$ so that if $u_1$, $u_2$ solve 
 \eqr{e:uQ}  for $t \in [0,1]$ with  
 \begin{align*}
   \| u_i \|_{C^1} \leq \delta \leq \delta_1, \, |\Hess_{u_i}| \leq C_2  {\text{ and }} 
  u_i (x,0)  = w_i (x) {\text{ for }}
 i=1,2 \, , 
 \end{align*}
   and $T$ is the linear map from \eqr{e:defT}, then 
 \begin{align}	 
	 \left\|   (u_1(x,1) - T(w_1) (x)) - (u_2(x,1) - T(w_2) (x)) 
	\right\|^2_{W^{1,2}}
	  \leq C_1 \, 
	   \delta^{\epsilon}  \,   \left\| w_1 - w_2 \right\|^2_{L^2}  \, . \notag
\end{align}
\end{Pro}

 \subsection{Lipschitz continuity of  $\cQ$}
  
  The next lemma shows that the nonlinearity $\cQ$  is Lipschitz with an arbitrarily small Lipschitz bound near $0$.  This is expected as the nonlinearity is higher order and thus, formally, its derivative at $0$ is zero.  We will give two formulations of this.  The first is an integral bound with a slightly better dependence, while the second is a pointwise bound that depends also on the second derivatives of the difference.
  
 \begin{Lem}	\label{l:cQLip}
There  exist $C$ and $\delta_0 > 0$ so that
if $u_1$ and $u_2$ have   $|u_i| + |\nabla u_i|   \leq \delta \leq  \delta_0$ 
and $\left| \Hess_{u_i} \right| \leq C_2$ for $i=1,2$ and $v$ is a function, then 
\begin{align}
	 \int_{\Sigma}
	   v    \, \left( \cQ(u_1) - \cQ(u_2) \right) \, \e^{ - \frac{|x|^2}{4} }  \leq  
	  C \, \delta \,  \int_{\Sigma}  
	 \left((1 + C_2) \,  |v  | + |\nabla v | \right) \,  \left( |u| + |\nabla u| \right)  
	    \, \e^{ - \frac{|x|^2}{4} }  
	      \, ,
\end{align}
where  $u(p) = u_1 (p) - u_2 (p)$ is the difference of the $u_i$'s. 
 Moreover, we have
\begin{align}	\label{e:comeback}
	\left| \cQ(u_1) - \cQ(u_2) \right| &\leq C(\delta + C_2) \,   \left( |u| + |\nabla u | \right)  + C \, \delta \,  |\Hess_u | \, .
\end{align}
\end{Lem}

\begin{proof}
Property (Q) from Lemma \ref{t:nonlinQ} gives that
\begin{align}	 
	v \, (\cQ(u_1) - \cQ(u_2))  
	 &=   v  \, 
	 	 \left( f +  \dv_{\Sigma} (W)  + \langle \nabla \bar{h}_{u_1} , V \rangle + \langle \nabla h , \bar{V}_{u_2} \rangle
		 + \langle \nabla (H\, u)  , \bar{V}_{u_2} \rangle
	 \right) \, , 
\end{align}
where $f$, $h$, $\bar{h}_{u_1}$, $V$, $\bar{V}_{u_2}$ and  $W$   are given by property (Q)
and $H$ is the mean curvature of $\Sigma$.
We will bound the integrals of each of the five terms on the right individually.

{\bf{The first term}}.  Using the bound 
$	|f| \leq   
	  C  \, \delta  \,   \left( |u| + |\nabla u | \right)$
 from property (Q) gives
\begin{align}
	\left| \int_{\Sigma}  v \, 
	 	 f   \, \e^{ - \frac{|x|^2}{4} }  \right| \leq  C  \, \delta  \,   \int_{\Sigma}  |v| \, 
	 	  \left( |u| + |\nabla u | \right)  \, \e^{ - \frac{|x|^2}{4} }  \, .
\end{align}

{\bf{The second term}}.  
We use 
 Stokes' theorem to take the derivatives  off of $W$  to get  
\begin{align}
	  \int_{\Sigma}    v  \,   \dv_{\Sigma} (W)  
	 \, \e^{ - \frac{|x|^2}{4} } &= 
	   \int_{\Sigma}  \left\{ - \langle \nabla   v \, , W \rangle + \frac{1}{2}
	  v \, \langle W , x \rangle \right\}
	 \, \e^{ - \frac{|x|^2}{4} }  	  \, .
\end{align}
Since $|x|$ is bounded on   $\Sigma$ and (Q) gives
$
	  |W|  \leq  	  C  \, \delta  \,   \left( |u| + |\nabla u | \right)$,
	  we get that
\begin{align}
	\left|   \int_{\Sigma}    v  \,   \dv_{\Sigma} (W)  
	 \, \e^{ - \frac{|x|^2}{4} }  \right| &\leq C \, \delta \, 
	   \int_{\Sigma} \left( |v| + | \nabla   v| \right) \, \left( |u| + |\nabla u| \right)
	 \, \e^{ - \frac{|x|^2}{4} }  	  \, .
\end{align}

{\bf{The third term}}.  Property (Q) gives  
$
	|V|  \leq C \, \left( |u| + |\nabla u| \right) $ and
	$\left| \nabla \bar{h}_{u_1} \right|  
	\leq C \, \delta \, (1 + C_2)$.
	This allows us to bound the third term by
\begin{align}
	\left|  \int_{\Sigma}  v\,   \langle \nabla \bar{h}_{u_1} ,  {V}  \rangle
	 \, \e^{ - \frac{|x|^2}{4} } \right| &\leq C \, \delta \, (1 + C_2)
	   \int_{\Sigma}   |v|   \, \left( |u| + |\nabla u| \right)
	 \, \e^{ - \frac{|x|^2}{4} }  	  \, .
\end{align}

{\bf{The fourth term}}. We use Stokes' theorem to take the derivative off of $h$ to get
\begin{align}
	  \int_{\Sigma}   v \,  \langle \nabla h , \bar{V}_{u_2} \rangle 
	 \, \e^{ - \frac{|x|^2}{4} } &= 
	   \int_{\Sigma}  \left\{ - \langle \nabla  v \, , h  \bar{V}_{u_2}  \rangle 
	   - v \, h \, \dv_{\Sigma} 
	   \left(
	    \bar{V}_{u_2} \right) 
	   + \frac{1}{2}
	  v \, \langle h \,\bar{V}_{u_2}   , x \rangle \right\}
	 \, \e^{ - \frac{|x|^2}{4} }  	  \, .
\end{align}
Since $|x|$ is bounded on   $\Sigma$ and    (Q) gives that
$	\left|  h    \right| \leq C \, \delta \, \left( |u| + |\nabla u| \right)$, 
	$\left|   \bar{V}_{u_2}  \right| \leq C \, \delta$,  and 
	 $|  \dv_{\Sigma} (\bar{V}_{u_2})| \leq C\,   (1 + C_2 )$, 
we bound the fourth term by
\begin{align}
	 \left|  \int_{\Sigma}   v \,  \langle \nabla h , \bar{V}_{u_2} \rangle 
	 \, \e^{ - \frac{|x|^2}{4} } \right|   &\leq
	  C \, \delta \,  \int_{\Sigma}   \left( |v| + | \nabla   v| \right) \, \left( |u| + |\nabla u| \right)
	 \, \e^{ - \frac{|x|^2}{4} }  \notag \\
	 &\qquad + C \, \delta \, C_2 
	  \int_{\Sigma}    |v|   \left( |u| + |\nabla u| \right)
	 \, \e^{ - \frac{|x|^2}{4} }	  \, .
\end{align}

{\bf{The fifth term}}. Since $|H| + |\nabla H|$ is bounded on the closed surface, we have
\begin{align}
	\left|   \int_{\Sigma}  v\,  \langle \nabla (H\, u)  , \bar{V}_{u_2} \rangle
	 \, \e^{ - \frac{|x|^2}{4} } \right| &\leq C  
	   \int_{\Sigma}   |v|   \, \left( |u| + |\nabla u| \right) \, \left| \bar{V}_{u_2}  \right|
	 \, \e^{ - \frac{|x|^2}{4} }  	  \, .
\end{align}
Using the bound   for $\left| \bar{V}_{u_2}  \right|$  bounds   the fifth term by
\begin{align}
	\left|   \int_{\Sigma}  v\,  \langle \nabla (H\, u)  , \bar{V}_{u_2} \rangle
	 \, \e^{ - \frac{|x|^2}{4} }  \right| &\leq C  \, \delta \, 
	   \int_{\Sigma}   |v|   \, \left( |u| + |\nabla u| \right)  
	 \, \e^{ - \frac{|x|^2}{4} }  	  \, .
\end{align}
This completes the proof of the integral bound.  

{\bf{The pointwise bound}}.  We argue as above for terms one, three and five, but we do not integrate by parts on terms two and four.  Instead, we use the last two conclusions from Lemma \ref{t:nonlinQ} on these terms.  Namely, we have
\begin{align}	 
	 |\dv_{\Sigma} (W) | &\leq C(\delta + C_2) \,    \left( |u| + |\nabla u | \right)  + C \, \delta \,  |\Hess_u |\,  , \\
	|\nabla h|   &\leq C(1 + C_2) \,  \left( |u| + |\nabla u | \right)  + C \, \delta \,  |\Hess_u | \, .  
\end{align}
Since $|\bar{V}_{u_2}| \leq C \, \delta$, we get
\begin{align}	 
	| \langle \nabla h , \bar{V}_{u_2} \rangle |  &\leq C(1 + C_2) \, \delta  \left( |u| + |\nabla u | \right)  + C \, \delta^2 \,  |\Hess_u | \, .  
\end{align}
The pointwise bound  now follows.
\end{proof}

\subsection{The map $\Psi$ is  continuous}

The next lemma shows that if two solutions of the nonlinear equation have initial values that are close in $L^2$, then they remain close in $W^{1,2}$.

\begin{Lem}	\label{l:L2bQ}
There  exist $C$ and $\delta_0 > 0$ so that
if $u_1$ and $u_2$ satisfy \eqr{e:uQ}  for $t \in [0,1]$ with   $|u_i| + |\nabla u_i|   \leq  \delta_0$ 
and $\left| \Hess_{u_i} \right| \leq C_2$ for $i=1,2$, then $u(x,t) = u_1 (x,t) - u_2 (x,t)$ satisfies
\begin{align}
	 \int_{\Sigma}	\label{e:19L2}
	 \left| u(x,t) \right|^2 \,\, \e^{ - \frac{|x|^2}{4} } &\leq \e^{C \, (1 + \delta_0 \, C_2^2)  t} \, \int_{\Sigma} \left| u (x,0)  \right|^2
	\, \e^{ - \frac{|x|^2}{4} }\, , \\
	  \int_0^1  \, \int_{\Sigma} \left| \nabla u (x,t)  \right|^2  \, \e^{ - \frac{|x|^2}{4} }
	&\leq  C \, (1 + \delta_0 \, C_2^2) \, \e^{C \, (1 + \delta_0 \, C_2^2) }  \,  \int_{\Sigma}  \left| u(x,0) \right|^2\, \e^{ - \frac{|x|^2}{4} }   \, .
\end{align}
Moreover, we also get 
 \begin{align}
   \| \nabla u (\cdot , 1) \|_{L^2}^2  +  \int_0^1 \int_{\Sigma}    |\Hess_u|^2   \, \e^{ - \frac{|x|^2}{4} }     &
	\leq       C \, (1 +   C_2^4) \, \e^{C \, (1 + \delta_0 \, C_2^2) }  \,  \int_{\Sigma}  \left| u(x,0) \right|^2\, \e^{ - \frac{|x|^2}{4} } 	  \, .
  \end{align}
\end{Lem}

 As an immediate consequence, we also get that $\Psi$ is continuous:

 \begin{proof}[Proof of Corollary \ref{l:ctypsi}]
 Let $u_1$ and $u_2$ be solutions as in Lemma \ref{l:L2bQ} and set $u = u_1 - u_2$.  
 Lemma \ref{l:L2bQ} gives  
 \begin{equation}	\label{e:kmj}
 	\| u (\cdot , 1 ) \|_{L^2} \leq C \, \| u (\cdot , 0 ) \|_{L^2} \, ,
 \end{equation}
 where $C$ is uniform as long as $u_1$ and $u_2$ are small in $C^1$ and bounded in $C^2$.  On the other hand, 
 Corollary \ref{c:mainpro} gives a uniform $C^{2, \alpha'}$ bound for $u_1$ and $u_2$, and hence also for $u$, for some $\alpha' > \alpha$.  
It follows from  interpolation inequalities that there is a $\beta$ in $(0,1)$ so that
\begin{equation}
	\| u (\cdot , 1 ) \|_{C^{2, \alpha}} \leq C \, \| u (\cdot , 1 ) \|^{\beta}_{L^2} \, \| u (\cdot , 1 ) \|^{1-\beta}_{C^{2, \alpha'}} \, .
\end{equation}
Finally, the corollary follows by combining this with \eqr{e:kmj}.
 \end{proof}

\begin{proof}[Proof of Lemma \ref{l:L2bQ}]
Within this proof, $C$ will be a constant that is allowed to change from line to line and depends only on  $\Sigma$ and an upper bound for  
$\| u_1 \|_{C^1} + \| u_2 \|_{C^1}$.

Recall that if $\phi$ is any function, then Stokes' theorem gives
$
	\int_{\Sigma} ( \cL \, \phi ) \, \e^{ - \frac{|x|^2}{4} } = 0$.
	Applying this with $\phi = u^2$ and using  that $L = \cL + |A|^2 + \frac{1}{2}$ gives
\begin{align}	 	\label{e:putbackin}
	\partial_t   \int_{\Sigma} u^2   \e^{ - \frac{|x|^2}{4} }  &=
	  \int_{\Sigma} \left( \partial_t - \cL  \right) \, u^2  \e^{ - \frac{|x|^2}{4} }
	 \leq \int_{\Sigma}  \left\{ C  u^2 - \left| \nabla u \right|^2 + u \left( \cQ(u_1) - \cQ(u_2) \right) \right\}  \e^{ - \frac{|x|^2}{4} }   \, . 
\end{align}   
Applying Lemma \ref{l:cQLip}  with $v=u$  bounds the $ u\, \left( \cQ(u_1) - \cQ(u_2) \right) $ term by
\begin{align}
	 \int_{\Sigma}
	   u    \, \left( \cQ(u_1) - \cQ(u_2) \right) \, \e^{ - \frac{|x|^2}{4} } & \leq  
	  C \, \delta_0 \,  \int_{\Sigma}  
	 \left((1 + C_2) \,  |u  | + |\nabla u | \right) \,  \left( |u| + |\nabla u| \right)  
	    \, \e^{ - \frac{|x|^2}{4} }  
	\notag \\
	& \leq  
	  C \, \delta_0 \,  \int_{\Sigma}  
	 \left(   \left( 1 + C_2^2 \right) \, u^2  
	+  |\nabla u|^2
	\right)
	    \, \e^{ - \frac{|x|^2}{4} }   \, ,
\end{align}
where the last inequality used   $(a+b) (c+d)\leq a^2 + b^2 + c^2 + d^2$.  
Putting this    into
\eqr{e:putbackin}
 gives
\begin{align}	\label{e:rbwwb2a}
	\partial_t    \int_{\Sigma} u^2 \, \e^{ - \frac{|x|^2}{4} } 
	  &\leq 
	    \int_{\Sigma}  \left\{    \left( C +C \, \delta_0 (1+ C_2^2) \right) \,
	    u^2 +  \left[  C \, \delta_0   - 2 \right]  \, |\nabla u|^2 
	       \right\}  \e^{ - \frac{|x|^2}{4} } \, .
	 \end{align}
  To control the energy term,  take $\delta_0 > 0$  small so that  
$ 
	   C \, \delta_0     \leq 1 $.   Using this in \eqr{e:rbwwb2a}  
gives
\begin{align}	\label{e:rbwwb2}
	\partial_t    \int_{\Sigma} u^2 \, \e^{ - \frac{|x|^2}{4} }   \leq 
	  C \, \left( 1 +    \delta_0 \, C_2^2 \right) \, \int_{\Sigma}  u^2 \, \e^{ - \frac{|x|^2}{4} } 
  -   \int_{\Sigma} \left| \nabla u \right|^2   \, \e^{ - \frac{|x|^2}{4} } 
	 \, .   
\end{align}
  The first claim follows by throwing away the last term   and integrating this differential inequality from $0$ to $t\leq1$.
To get the second claim, we integrate \eqr{e:rbwwb2} in  time  to get
\begin{align}
	\int_{t=1}  u^2 \, \e^{ - \frac{|x|^2}{4} } 
 -  \int_{t=0} u^2\,   \e^{ - \frac{|x|^2}{4} } 
 &\leq  
	C  \left( 1 + \delta_0 \, C_2^2 \right) \, \int_0^1 \, \int_{\Sigma}  u^2  \e^{ - \frac{|x|^2}{4} }   -   \int_0^1  \, \int_{\Sigma} \left| \nabla u \right|^2  \e^{ - \frac{|x|^2}{4} } 
 \, .
\end{align}
Combining this with
  the first claim gives the second claim.

  We turn next to   the higher derivative bounds.   We have
  \begin{align}
  	\frac{1}{2} \, \partial_t \, \int_{\Sigma} |\nabla u|^2 \, \e^{ - \frac{|x|^2}{4} }  &= \int_{\Sigma} \langle \nabla u_t , \nabla u \rangle \, \e^{ - \frac{|x|^2}{4} } 
	= -  \int_{\Sigma}    u_t \cL  u  \, \e^{ - \frac{|x|^2}{4} }  \notag \\
	&= -  \int_{\Sigma}    \left( \cL u  + (|A|^2 + 1/2) u +  \left( \cQ(u_1) - \cQ(u_2) \right) \right) \, \cL  u  \, \e^{ - \frac{|x|^2}{4} } \\
	&\leq    \int_{\Sigma}    \left\{ - \frac{1}{2} (\cL u)^2  + C \, |\nabla u|^2 + u^2 +  \frac{1}{2}\, \left( \cQ(u_1) - \cQ(u_2) \right)^2 \right\}   \, \e^{ - \frac{|x|^2}{4} } 
	\notag \, .
  \end{align}
  The drift Bochner formula and the divergence theorem give that
  \begin{align}
  	\int_{\Sigma}  |\Hess_u|^2 \, \e^{ - \frac{|x|^2}{4} } 
 	\leq \int_{\Sigma}  \left\{ (\cL u)^2 + C \, |\nabla u|^2 \right\} \, \e^{ - \frac{|x|^2}{4} } \, , 
  \end{align}
  so we get that 
   \begin{align}
    \partial_t \, \int_{\Sigma} |\nabla u|^2 \, \e^{ - \frac{|x|^2}{4} }  &
	\leq    \int_{\Sigma}    \left\{ -   |\Hess_u|^2  + C \, |\nabla u|^2 + 2\, u^2 +    \left( \cQ(u_1) - \cQ(u_2) \right)^2 \right\}   \, \e^{ - \frac{|x|^2}{4} } 
	  \, .
  \end{align}
  The last part of Lemma \ref{l:cQLip}
gives  
\begin{align}	 
	\left| \cQ(u_1) - \cQ(u_2) \right|^2 &\leq C( 1 + C_2^2) \,   \left( |u|^2 + |\nabla u |^2 \right)  + C \, \delta^2 \,  |\Hess_u |^2 \, .
\end{align}
Taking $\delta_0 > 0$ small enough that $C \, \delta^2 < \frac{1}{2}$, it follows that
   \begin{align}
   \int_{\Sigma}  \frac{ |\Hess_u|^2 }{2}  \, \e^{ - \frac{|x|^2}{4} }  +  \partial_t \, \int_{\Sigma} |\nabla u|^2 \, \e^{ - \frac{|x|^2}{4} }  &
	\leq     C( 1 + C_2^2) \,   \int_{\Sigma}    \left( |u|^2 + |\nabla u |^2 \right)     \, \e^{ - \frac{|x|^2}{4} } 
	  \, .
  \end{align}
  Integrating this in $t$ and using the first two claims to bound the right-hand side completes the proof.
\end{proof}

 \subsection{The $Q$-Lipschitz approximation property}

We will next prove the $Q$-Lipschitz approximation property for the time-one rescaled MCF in a neighborhood of a  shrinker $\Sigma$.
Namely,   the time one flow is $Q$-Lipschitz close to the linear  mapping $T$ defined in \eqr{e:defT}.

\begin{proof}[Proof of Proposition \ref{p:lips}]
Let $\tilde{w}_i$ solve  $
	\left( \partial_t - L \right) \, \tilde{w}_i = 0$  with 
$
 \tilde{w}_i (x,0) = w_i (x) $, so that  $\tilde{w}_i (x,1) = T(w_i)(x)$.  Set
 $v_i = u_i - \tilde{w}_i$.  It follows that $v_i (x,0)  = 0 $
 \begin{align}
 	v_i (x,1) &= u_i (x,1) - T(w_i)(x) \, , \\
	\left( \partial_t - L \right) \, v_i &= \cQ(u_i)  \, .
\end{align}
Finally, let $v$ and $u$ be the differences of the $v_i$'s and $u_i$'s, i.e., 
\begin{align}
	v(x,t) &= v_1 (x,t) - v_2 (x,t) \, , \\
	u(x,t) &= u_1 (x,t) - u_2 (x,t) \, , 
\end{align}
and define $\psi(t)$ to be the $L^2$ norm (squared) of $v$ at time $t$
\begin{equation}
	\psi(t) =   \int_{\Sigma}  \left|   v (x,t) \right|^2 \, \e^{ - \frac{|x|^2}{4} } \, .
\end{equation}
To prove \eqr{e:19L2}, we will get a uniform bound for  $\psi (t)$ for all $t\leq 1$.

 We will derive a differential inequality for $\psi (t)$.  Applying $\partial_t - \cL$ to $v^2$ as in \eqr{e:putbackin} gives
\begin{align}	\label{e:pski}
	\frac{1}{2} \, \psi ' (t) &
	 \leq \int_{\Sigma}  \left\{ C \, v^2  - \left| \nabla v  \right|^2  + v \,  (\cQ(u_1) - \cQ(u_2))     \right\}
	 \, \e^{ - \frac{|x|^2}{4} }   
	\, .
\end{align}
Applying Lemma \ref{l:cQLip}, we bound the $v \,  (\cQ(u_1) - \cQ(u_2)) $ term by
\begin{align}	
	& \int_{\Sigma}  v \, 
	 	 (\cQ(u_1) - \cQ(u_2))   \, \e^{ - \frac{|x|^2}{4} }    \leq C \, \delta \,  \int_{\Sigma}  
	 \left((1 + C_2) \,  |v  | + |\nabla v | \right) \,  \left( |u| + |\nabla u| \right)  
	    \, \e^{ - \frac{|x|^2}{4} }     \notag \\
	   &\qquad \qquad  \leq C \, \delta   \,  \int_{\Sigma}  
	\left( u^2 + |\nabla u|^2 + (1 + C_2)^2 \, v^2 + |\nabla v|^2 \right)  
	    \, \e^{ - \frac{|x|^2}{4} } \, , 
\end{align}
where the last inequality used the   inequality $(a+b) (c+d)\leq a^2 + b^2 + c^2 + d^2$.  
  Substituting this  bound back into \eqr{e:pski}, we get that
 \begin{align}
	  \psi ' (t) &\leq 
 	 C \, \left( 1 +   \delta   \, C_2^2
			 \right) \int_{\Sigma}   v^2    \, \e^{ - \frac{|x|^2}{4} } + \left( C \, \delta   - 2 \right) \, \int_{\Sigma} 
			 |\nabla v|^2 \, \e^{ - \frac{|x|^2}{4} } \notag \\
			 &\qquad + C \, \delta \,  \int_{\Sigma}  \left( |u|^2 + |\nabla u |^2 \right) 
			 \, \e^{ - \frac{|x|^2}{4} } \, .
\end{align}
We now choose $\delta_1 > 0$ so that $C \delta_1  \leq 1$ and $|\nabla v|^2$ term is negative.  We get that
 \begin{align}
	  \psi ' (t) + \int_{\Sigma}  |\nabla v|^2 \, \e^{ - \frac{|x|^2}{4} }  &\leq 
 	 C \, \left( 1 +  \delta   \, C_2^2 
			 \right)  \psi (t)  + C \,  \delta   \, \int_{\Sigma}  \left( |u|^2 + |\nabla u |^2 \right) 
			 \, \e^{ - \frac{|x|^2}{4} } \, .
\end{align}
To simplify notation, set $\kappa =  C \, \left( 1 +  \delta   \, C_2^2 
			 \right)$, so we get the differential inequality
 \begin{align}
 	\left( \e^{-\kappa\, t} \, \psi  \right)' + \int_{\Sigma}  |\nabla v|^2 \, \e^{ - \frac{|x|^2}{4} }   \leq 
	  C \, \delta  \, 
	   \int \left(    \left|u  \right|^2 + \left| \nabla u  \right|^2 \right) \, \e^{ - \frac{|x|^2}{4} }    \, ,
\end{align}
Integrating this up in time and using that $\psi (0) = 0$ gives
 \begin{align}
 	   \e^{-\kappa} \, \sup_{s\in [0,1]}   \psi   (s) + \int_0^1 \int_{\Sigma}  |\nabla v|^2 \, \e^{ - \frac{|x|^2}{4} } 
	  \leq 
	  C \, \delta \, 
	  \int_0^1  \int_{\Sigma}  \left(    \left|u  \right|^2 + \left| \nabla u  \right|^2 \right) \, \e^{ - \frac{|x|^2}{4} }    \, .
\end{align}
Using Lemma \ref{l:L2bQ} to bound the right-hand side gives 
 \begin{align}	\label{e:L2vbound}
 	   \e^{-\kappa} \, \sup_{s\in [0,1]}   \psi   (s) + \int_0^1 \int_{\Sigma} \left\{ v^2 + |\nabla v|^2\right\}  \, \e^{ - \frac{|x|^2}{4} } 
	  \leq 
	  C \, \delta \, 
	  \int_{\Sigma} |u (\cdot , 0)|^2 \, \e^{ - \frac{|x|^2}{4} }   \, .
\end{align}

We turn next to proving a $W^{1,2}$ on $v$ at time $1$.  Define $\chi (t)$ to be the $L^2$ norm squared of $\nabla v$ at time $t$
\begin{align}
	\chi (t) =\int_{\Sigma}  \left|  \nabla  v (x,t) \right|^2 \, \e^{ - \frac{|x|^2}{4} } \, .
\end{align}
Using the space-time $L^2$ bound on $|\nabla v|$ that we have already, there exists $t_0 \in [1/2 , 1]$ with
\begin{align}	\label{e:chit0}
	\chi (t_0) \leq C \, \delta \, \int_{\Sigma} |u (\cdot , 0)|^2 \, \e^{ - \frac{|x|^2}{4} }  \, .
\end{align}
To get the bound on $\chi (1)$, we will bound the integral of $\chi'$ from $t_0$ to $1$.

The divergence theorem and the equation $v_t = \cL v + (|A|^2 + 1/2) v + \cQ(u_1) - \cQ(u_2)$ give
\begin{align}
	\frac{1}{2} \, \chi' (t) &=\int_{\Sigma} \langle \nabla v_t   ,   \nabla  v \rangle \, \e^{ - \frac{|x|^2}{4} } = 
	- \int_{\Sigma}  v_t   \,  \cL  v   \, \e^{ - \frac{|x|^2}{4} } \notag \\
	&=  - \int_{\Sigma}   \cL  v \left\{   \cL v + (|A|^2 + 1/2) v + \cQ(u_1) - \cQ(u_2) \right\}   \, \e^{ - \frac{|x|^2}{4} }  \\
	&=   \int_{\Sigma}     \left\{   (|A|^2 + 1/2 ) |\nabla v|^2 + v \langle \nabla |A|^2 , \nabla v \rangle
	 - \cL v \left[ \cQ(u_1) - \cQ(u_2) \right]   - (\cL v)^2  \right\}   \, \e^{ - \frac{|x|^2}{4} }
	\, . \notag
\end{align}
Bounding the first two terms on the right in terms of $v^2$ and $|\nabla v|^2$ and using an absorbing inequality gives
\begin{align}
	 \chi' (t) &\leq    \int_{\Sigma}     \left\{   c \,  |\nabla v|^2 + v^2
	 + \left[ \cQ(u_1) - \cQ(u_2) \right]^2  \right\}   \, \e^{ - \frac{|x|^2}{4} }
	\, . \notag
\end{align}
where the constant $c$ depends only $\| |A|^2 \|_{C^1}$.     The last part of Lemma \ref{l:cQLip}
gives    
\begin{align}	 
	\left| \cQ(u_1) - \cQ(u_2) \right|^2 &\leq C( \delta + \sup_i | \Hess_{u_i} (\cdot , t)|^2) \,   \left( |u|^2 + |\nabla u |^2 \right)  + C \, \delta^2 \,  |\Hess_u |^2 \, .
\end{align}
  Interior (in time) parabolic Schauder estimates and interpolation give $\epsilon > 0$ so that 
  \begin{align}
  	 \sup_{t \in [1/2,1]} | \Hess_{u_i} (\cdot , t)|^2 \leq C \, \delta^{\epsilon} \, .
\end{align}
Thus, we get that
\begin{align}
	\int_{t_0}^1 \chi'(t) \, dt  &\leq  C \, 
		\int_{t_0}^1 \int_{\Sigma} \left\{ \delta^{\epsilon} \,  \left( u^2 + |\nabla u|^2 \right)  + \delta^2 \, |\Hess_u |^2  + \left( v^2 + |\nabla v|^2 \right) \right\} \, \e^{ - \frac{|x|^2}{4} }  \notag \\
		&\leq C \, \delta^{\epsilon} \,  \int_{\Sigma} |u (\cdot , 0)|^2 \, \e^{ - \frac{|x|^2}{4} }   \, ,  
\end{align}
where the last inequality used Lemma \ref{l:L2bQ} and   \eqr{e:L2vbound}.   Combining this with \eqr{e:chit0} completes the proof.
\end{proof}

\section{The action of the rotation group}

The rotation group $\cR$ acts on $u \in E$ by applying a Euclidean rotation to the graph of $u$ over $\Sigma$.   If the rotation is far from the identity (and $\Sigma$ is not a sphere), then the new hypersurface may no longer be written as a graph over $\Sigma$ and, in particular, the action does not necessarily preserve $E$.

In this section, we will show that the action of the rotation group $\cR$ satisfies the properties ($\cR 0$)--($\cR 2$).  The first property ($\cR 0$) is automatic since rotations preserve both the geometry and the Gaussian weight.  The other two properties require some work.  Properties ($\cR 1$) and ($\cR 2$) are given by the next proposition:

\begin{Pro}	\label{p:cRact}
Given $\epsilon > 0$, there exists $\delta > 0$ so that if $u,v \in E$ and $g \in \cR$ satisfy
\begin{align}
	\|u\|_E , \|v\|_E , \|g(u)\|_E   < \frac{\delta}{3} \, , 
\end{align}
then $\|g(v)\|_E < \delta$ and
\begin{align}
	\|g(v) - g(u)\|_Q &\leq (1+\epsilon) \, \|v-u\|_Q \, , \label{e:gyx1} \\
	\|v_1 - u_1\|_Q &\leq \|(g(v)-g(u))_1 \|_Q + \epsilon \, \|v-u\|_Q \, .  \label{e:gyx2} 
\end{align}
\end{Pro}

The proposition will follow easily from the next lemma that writes the graph of $v$ as a normal graph over the graph of $u$.

\begin{Lem}	\label{l:graph}
There exist $C$ and $\delta_0 > 0$ so that if $u , v \in E$ satisfy $\|u\|_E , \|v\|_E < \delta_0$, then the graph of $v$ can be written as a normal graph over the graph of $u$ of a function $w$ satisfying the pointwise estimate for $p \in \Sigma$
\begin{align}	\label{e:graph}
	|w(p) - (v-u)(p)| \leq C \, \delta_0^2 \, |(v-u)(p)| \, .
\end{align}
\end{Lem}

\begin{proof}
Since $u,v$ and their gradients are small and $A$ is bounded on $\Sigma$ and both graphs, we  
get the existence of a normal graph function $w$.  The point is to establish the estimate \eqr{e:graph}.  

Following the appendix in \cite{CM3}, define the mapping $B(p,s)= I - s \, A_p$ on the tangent space to $\Sigma$ at $p$ so that the vector field 
\begin{align}
	V(p) \equiv \nn (p) - B^{-1}(p,u(p)) (\nabla u(p))
\end{align}
 is normal{\footnote{Note that $V(p)$ is not the unit normal, but its norm is one up to higher order corrections.}}
  to the graph of $u$ at the point $p + u(p) \, \nn(p)$.  The function $w(p)$ is the length of the segment that leaves $p + u(p) \, \nn(p)$ in the direction $V(p)$ and intersects the graph of $v$.  Define $q=q(p) \in \Sigma$ to be the point so that the segment ends at $q + v(q) \, \nn(q)$.  We are looking to solve for $q,s$ satisfying
 \begin{align}	\label{e:spq}
 	p+u(p) \, \nn(p) + s\, V(p) = q+ v(q) \, \nn (q) \, .
 \end{align}
 We can rewrite this as
  \begin{align}	\label{e:spq2}
 	p-q &+ s \, B^{-1}(p,u(p)) (\nabla u(p))  = v(q) \, \nn (q) - u(p) \, \nn(p) - s \, \nn (p) \\
	&= (v(q)-v(p)) \, \nn (q) +v(p) \,(\nn(q)- \nn(p)) + (v(p)  - u(p) -s)\, \nn(p)   \, .  \notag
 \end{align}
 Taking the tangent (at $p$) part of these vectors, we see that
 \begin{align}
 	|(p-q)^T| \leq C \, \delta_0^2 \, s+ C \, \delta_0 \, |p-q| \, .
 \end{align}
 Since $p$ and $q$ are close in $\Sigma$, the difference is almost tangent and we conclude that
  \begin{align}	\label{e:pqclose}
 	|p-q| \leq C \, \delta_0^2 \, s \, .
 \end{align}
  This time we take the normal part in \eqr{e:spq2} to get
    \begin{align} 
    	\left|  (v(p)  - u(p) -s) \right| &\leq 
 	|\langle (p-q), \nn (p) \rangle| +  |v(q)-v(p)| + |v(p)| \,|\nn(q)- \nn(p)|  
	\, .   
 \end{align}
 Using \eqr{e:pqclose} on the right hand side gives $\left|  (v(p)  - u(p) -s) \right| \leq C \, \delta_0^2 \, s$.  The claim follows easily from this.
\end{proof}

\begin{proof}[Proof of Proposition \ref{p:cRact}]
Let $U$ be a neighborhood  of $\Sigma$ (to be chosen).   Given $u\in U$, let $L_u$ be the second variation operator on $\Sigma_u$, $Q_u$ the induced Gaussian 
$W^{1,2}$ inner product (that makes $L_u$ self-adjoint),  $ \mu_u$ be largest eigenvalue of $L_u$ that is less than $-1$, and $E_u^1$ the span of the eigenfunctions of $L_u$ with eigenvalues less than $-1$.
 Let   $\Pi^u_1$ be $Q_u$-orthogonal projection to $E^1_u$.  Moreover,  $\mu_u$,  $E^1_u$, and $\Pi^u_1$   are continuous{\footnote{The continuity of the union of eigenspaces relies on the gap to $-1$.}}  as long as $U$ is small enough.
 
 The proposition now follows in three steps.  First, Lemma \ref{l:graph} gives a function $w$ so that the graph of $v$ is a normal graph of $w$ over the graph of $u$ and $|w-(v-u)| \leq C \, \delta^2 \, |v-u|$ pointwise.  It follows that
 \begin{align}	\label{e:qz1}
	\left|	\|w\|_{Q_u}-  \|v-u \|_Q \right| &\leq (1+\epsilon_1) \, \|v-u\|_Q \, , \\
	\|v_1 - u_1\|_Q &\leq \|\Pi^u_1 (w) \|_{Q_u} + \epsilon_1 \, \|v-u \|_Q \, .   \label{e:qz2}
\end{align}
If we now apply $g$ to $u$ and $v$, then the graph of $g(v)$ is a normal graph over the graph of $g(u)$ of the same function $w$.  To be precise, the function $w$ is unchanged on the underlying manifold (the graph of $u$, which is isometric to the graph of $g(u)$), but there is a new identification between points in the graph and points in $\Sigma$. The operator $L_u$ is also preserved by the action of $g$ (we use here that $g$ is a rotation about the origin, so it also preserves the Gaussian weight).  It follows that \eqr{e:qz1} and \eqr{e:qz2} hold with $Q_u$ replaced by $Q_{g(u)}$ and $\Pi^u_1$ replaced by
$\Pi^{g(u)}_1$.
   We can now apply Lemma \ref{l:graph} in the reverse direction 
to relate  $w$ and $g(v) - g(u)$, completing the proof.
\end{proof}

\section{Hypersurfaces modulo translations, dilations and rotations}	\label{s:groupo}

In this section, we will complete the proof of the main theorem by analyzing an equivalent dynamical system that mods out the action by dilations and translations.  The rescaled MCF is the gradient flow for the $F$ functional and, thus, builds in a choice of a center and scale.  Following \cite{CM1},  
 the entropy $\lambda$ of a hypersurface $M \subset \RR^{n+1}$ mods out for this choice    by taking the supremum of   Gaussian areas over all possible centers and scales
\begin{align}	\label{e:entrop}
	\lambda (M) = \sup_{(x_0 , t_0) \in \RR^{n+1} \times \RR^+ } \, \, F_{x_0 , t_0} (M) \, , 
\end{align}
where $F_{x_0 , t_0} (M) $ is the Gaussian area with center $x_0$ and scale $t_0$ given by
\begin{align}
	F_{x_0 , t_0} (M) = \left( 4 \, \pi \, t_0 \right)^{ - \frac{n}{2} } \, \int_M \e^{ - \frac{ |x-x_0|^2}{4t_0} } \, .
\end{align}

We will say that $M$ is {\emph{balanced}} if its entropy is equal to its $F=F_{0,1}$ functional. 
 By lemma $7.10$ in \cite{CM1}, any shrinker is automatically balanced.
 Let $\Gamma \subset U_0$ be the set of balanced graphs, i.e., $u \in U_0$ is in $\Gamma$ if the graph $\Sigma_u$ satisfies the balancing condition
\begin{align}	\label{e:balance}
	F(\Sigma_u) = \lambda (\Sigma_u) \, .
\end{align}

Heuristically, the way to mod out for translations and dilations would be to look at the gradient flow for $\lambda$.  However, $\lambda$ is not in general differentiable
since it is given as a supremum.  To get around this, 
we will   analyze the dynamics on $\Gamma$.  The key idea is that each graph $\Sigma_u$ nearby $\Sigma$ has a unique center of mass and scale that achieve its entropy.  Thus, there is a canonical translation and dilation that ÒbalancesÓ it to have center $0$ and scale $1$ and this ``balancing mapÕÕ $\gamma$ depends smoothly on $u$. 
When we relate this to the original dynamics, it will be crucial that the balancing map commutes with rescaled MCF.

\subsection{The balancing map}
 
Let $\cG$ be the group generated by translations and dilations of $\RR^{n+1}$.  The group $\cG$ can be parameterized by $(y,h) \in \RR^{n+1} \times \RR^+$, where we associate $(y,h)$ to the map $g_{y,h}$ given by
\begin{align}
	g_{y,h} (p) = h(p) + y \, .
\end{align}
Let $\cg \subset E_2$ be the linear space of translation and dilation vector fields
\begin{align}
	\cg = \{ y^{\perp} + b \, x^{\perp} \, | \, y \in \RR^{n+1} , b \in \RR \} \, .
\end{align}
 The translations   lie in the $-\frac{1}{2}$ eigenspace of $L$ while the dilations are in the $-1$ eigenspace.  Let $\cg^{\perp}$ denote the orthogonal complement of $\cg$ with respect to the inner product $Q$.

The next proposition gives a ``balancing map'' $\gamma$ that maps each graph near $\Sigma$ to a nearby balanced graph,  
does so in a $Q$-Lipschitz way, and is the identity on $\Gamma$.   
Let $\cT$ be the $Q$-orthogonal projection from $E$ to $\cg^{\perp}$.  Obviously, $\cT$ is linear.  

\begin{Pro}	\label{t:balancing}
There exists $\delta_c > 0$, $C$, and a map $\gamma: B_{\delta_c} \subset E \to B_{2\delta_c} \cap \Gamma$ so that $\gamma$ is the identity on $\Gamma$ and if $\delta \leq \delta_c$ and
$u,v \in B_{\delta} \subset E$, then 
\begin{align}	\label{e:QLipgamma}
	\| (\gamma (u) - \cT (u)) - (\gamma (v) - \cT (v)) \|_Q \leq C\, \delta \, \| u - v \|_Q \, .
\end{align}
\end{Pro}

 \subsection{Center of mass and the proof of Proposition \ref{t:balancing}}
 
 The next lemma shows that the optimal center and scale in \eqr{e:entrop} is Lipschitz continuous with respect to the Gaussian $W^{1,2}$ norm.  
 The distance on 
 $\RR^{n+1} \times \RR^+$ is defined to be
 \begin{align}
 	\| (x_0 , t_0) - (y_0 , s_0) \|_{\RR^{n+1} \times \RR^+} \equiv |x_0 - y_0| + |\log t_0 - \log s_0 | \, .
 \end{align}
 The linear map $T_{\rho}$ in the lemma is the linearization (or derivative) of $\rho$.

 \begin{Lem}	\label{p:COM}
 There exists $\delta_b > 0$, $C$,  a map $\rho: B_{\delta_b} \subset  E \to \RR^{n+1} \times \RR^+$ and a linear map $T_{\rho}: E \to \RR^{n+1} \times \RR$  so that:
 \begin{enumerate}
 \item[($\rho 1$)]  $\lambda (\Sigma_u) = F_{\rho (u)} (\Sigma_u)$ and, thus, $\rho = (0,1)$ on $\Gamma$.
 \item[($\rho 2$)]  If $v \in W^{1,2}$, then $| T_{\rho} (v)| \leq C \, \| v \|_{W^{1,2}}$.  If $v \in \cg^{\perp}$, then $T_{\rho} (v) = 0$.
 \item[($\rho 3$)]   If $\delta < \delta_b$ and $u,v \in B_{\delta} \subset  E$, then
 \begin{align}
 	\| \rho (u) - \rho (v) - T_{\rho}(u-v) \|_{\RR^{n+1} \times \RR^+} \leq C \, \delta \, \| u -v \|_{W^{1,2}} \, .
 \end{align}
 \end{enumerate}
 \end{Lem}
 
 \begin{proof}
The key is to examine the map 
 \begin{align}
 	G(y_0 , s_0 , u) = F_{y_0 , s_0} (\Sigma_u) \, .
\end{align}
In section $7$ of \cite{CM1}, it is proven that $(y_0 , s_0) \to G(y_0 , s_0 , 0)$ has a strict maximum at $(0,1)$ and its Hessian there is negative definite (this uses that $\Sigma$ cannot split off a line since it is compact).
 Moreover, it follows from the proof of theorem $0.15$ in \cite{CM1} (in section $7$ there) that, as long as $\delta_b > 0$ is small, the optimal center and scale in \eqr{e:entrop} can be achieved only in a small ball around $(0,1)$.    
 
  If we fix $u$, then the derivative of the map $(y_0 , s_0) \to G(y_0,s_0, u)$  is given by the vector-valued function (see lemma $3.1$ in \cite{CM1})
\begin{align}	\label{e:bF}
	\bF (x_0 , t_0 , u ) = \left( 4 \, \pi \, t_0 \right)^{ - \frac{n}{2} } \, \int_{\Sigma_u} \left\{    \left( \frac{|x-x_0|^2 - 2nt_0}{4t_0^2} \right) , \frac{x-x_0}{2t_0}
	\right\} \e^{ - \frac{ |x-x_0|^2}{4t_0} } \, .
\end{align}
Observe that we can write  $\bF$ as
\begin{align}	\label{e:bFu}
	\bF (x_0 , t_0 , u ) =   \left( 4 \, \pi \, t_0 \right)^{ - \frac{n}{2} } \int_{\Sigma}  \nu (p, u(p) , \nabla u (p)) \, \Xi (p, u(p)) \e^{ - \frac{ |x-x_0|^2}{4t_0} } \, ,
\end{align}
where $\nu$ is the relative area function from the appendix and $\Xi$ is a vector-valued function of $p$ and $u(p)$. In particular, $\bF$ depends only on $x_0 , t_0$, the value of $u$, and   $\nabla u$.  

We will use the implicit function theorem to get the map $\rho$.  To do this, we need to understand the derivative of $\bF$ both with respect to $(x_0,t_0)$ and with respect to $u$.   Since $\bF$ is itself the $(x_0, t_0)$ derivative of $G$,   it follows that the $(x_0, t_0)$ derivative $d_{(x_0,t_0)} \bF$ of $\bF$ is the  second derivative of $G$ in the $(x_0, t_0)$ direction.  Thus, since $\Sigma$ does not split off a line,  
 \cite{CM1} implies that $d_{(x_0,t_0)} \bF$ is invertible at $0$.  By continuity in $u$, $d_{(x_0,t_0)} \bF$ is invertible in a ball about $0$.  Next, if we differentiate
 \eqr{e:bFu} along a path $u+tv$, then the chain rule gives
 \begin{align}
 	\frac{d}{dt}\big|_{t=0} \, \bF (x_0 , t_0 , u+tv) &=  \left( 4 \, \pi \, t_0 \right)^{ - \frac{n}{2} } \int_{\Sigma}  \nu (p, u , \nabla u ) \, \Xi_{s} (p, u) \, v \e^{ - \frac{ |x-x_0|^2}{4t_0} } \notag \\
	&+  \left( 4 \, \pi \, t_0 \right)^{ - \frac{n}{2} } \int_{\Sigma}  \left[ \nu_s (p, u , \nabla u )\, v  +  \nu_{y_{\alpha}} (p, u , \nabla u )  \, v_{\alpha}  \right] \, \Xi (p, u) \e^{ - \frac{ |x-x_0|^2}{4t_0} }
\, .
 \end{align}
 It follows that the linear map $d_u \bF$ can be written as
 \begin{align}
 	d_u \bF(v) =  \int_{\Sigma} \zeta_1 (x_0 , t_0 , p , u (p) , \nabla u (p)) \, v +  \int_{\Sigma} \zeta_{2,\alpha} (x_0 , t_0 , p , u (p) , \nabla u (p)) \, v_{\alpha} \, , 
 \end{align}
 where $\zeta_1 , \zeta_{2,\alpha}$ are smooth vector-valued functions.
Thus, we can apply the implicit function theorem, $1.5$ in \cite{HiP}, to get $\delta_b > 0$ and a map $\rho: B_{\delta_b} \subset  E  \to \RR^{n+1} \times \RR^+$ so that
 \begin{align}
 	\bF (\rho (u) , u ) &= 0 \, , \\
	d_u \rho &= - d^{-1}_{(x_0 , t_0)} \circ d_u \bF \, .
\end{align}
Define the linear map $T_{\rho}$ to be $d_u \rho$ at $x_0 = 0$, $t_0 =1$ and $u=0$.
  If $v \in \cg^{\perp}$, then section $4$ in \cite{CM1} gives that $T_{\rho} (v)=0$, giving ($\rho 2$).  Finally,  ($\rho 3$) follows easily from the form of $d \bF$.
     \end{proof}

 \begin{proof}[Proof of Proposition \ref{t:balancing}]
 Let the map $\rho$ be given by
 Lemma \ref{p:COM}.  This induces a map 
 \begin{align}
 	\bar{\rho} :   B_{\delta_b} \subset E \to \cT \, , 
\end{align}
 where $\bar{\rho} (u)$ translates and dilates to take $\rho (u)$ to $(0,1)$.  
The balancing map $\gamma$ is then defined by letting $\bar{\rho} (u)$ act on $u$.  
Since $\rho (u) = (0,1)$ if $u \in \Gamma$,  $\gamma$ is the identity on $\Gamma$.

We will show next that $\cT$ is the linearization (or derivative) of $\gamma$ at $0$, using different arguments to compute the linearization first in the direction of $\cg^{\perp}$ and then in the direction of $\cg$.
 Property ($\rho 2$) in Lemma \ref{p:COM} implies that the linearization of $\gamma$ at $0$ is the identity on $\cg^{\perp}$.   Since the action of the group is undone by $\gamma$, it follows that $\cT$  is the linearization of $\gamma$ at $0$.   
 
 Finally, property ($\rho 3$) in Lemma \ref{p:COM} gives the 
  the $Q$-Lipschitz property \eqr{e:QLipgamma}.
 \end{proof}

 \subsection{The proof of Theorem \ref{t:one}}
 
 We will use Proposition \ref{t:balancing} to complete the proof of the main theorem of the paper. 
  Let $\Psi$ be the time one map for rescaled MCF, restricted to a small neighborhood $U_0 \subset E$ of $0$, and $T$ its linearization.  As before, we have a $T$-invariant
splitting $E=E_1 \oplus E_2$ where $T$ is strictly expanding on $E_1$ and less expanding on $E_2$.

    The key will be to mod out the group action by considering an equivalent dynamical system on the set of balanced hypersurfaces $\Gamma$.   To do this, define $\PG: U_0   \to \Gamma$ by 
 \begin{align}
 	\PG (u) = \gamma (\Psi (u)) \, .
 \end{align}
 To make this work, it will be crucial that:
 \begin{itemize}
 \item $\gamma$ commutes with $\PG$.
 \item The rotation group $\cR$ commutes with $\PG$.
 \end{itemize}

\begin{proof}[Proof of Theorem \ref{t:one}]
 Let $s>0$ be a small constant to be chosen and let $W_0$ be the set of points whose trajectories never leave the (closed)  $s$-tubular neighborhood of the orbit $\cR_0$ under the action of $\PG$
\begin{equation}
	W_0 = \{ x \in E \, | \, {\text{ for all $n\geq 0$ there exists  $g_n \in \cR$ so that }} g_n(\PG^n (x)) \in \overline{B_s}  \} \, .
\end{equation}
Since $\PG$ and the action are continuous, $W_0$ is closed.

We will apply the general results from Section \ref{s:one} to the map $\PG$.  
It follows from the chain rule   that the linearization $T_{\Gamma}$ of $\PG$ at $0$ is the composition of $\cT$ and $T$.  This preserves the splitting of $E_1$ and $E_2$ and satisfies property (2) from Section \ref{s:one}.  Property (3) in Section \ref{s:one} follows from  property (3) for $\Psi$ and $T$ together with Proposition \ref{t:balancing} and  the triangle inequality.  
Since $\cR$ commutes with $\PG$, properties ($\cR 0$)--($\cR 2$) for $\Psi$ extend to $\PG$.  Consequently, Proposition \ref{l:w2cR} applies and, thus, 
if $s> 0$ is sufficiently small, then $B_s \cap W_0$ is  the graph of a $Q$-Lipschitz mapping $u: P_2(W_0) \subset E_2 \to E_1$.  

Finally, we will show that the complement of $W_0$ has the desired properties.  Suppose therefore that $v \in B_s \setminus W_0$.  By the definition of $W_0$, there is some first positive integer $n$ so that if $g$ is any rotation, then $g (\PG^n (v)) \notin \overline{B_s}$.   Note that $g_{n-1} (\PG^{n-1} (v))$ is in $B_s$, so 
$g_{n-1} (\PG^n (v))$ is in a small ball $B_{s'}$ (by continuity) and, by construction, is also in $\Gamma$.  
The hypersurface $\PG^n (v)$ differs from $\Psi^n (v)$ by a translation and dilation. 

Define the set $\Omega =  \overline{B_{s'}} \cap \Gamma \setminus \cR({B_{s}})$ and then let $\Omega^B \subset \Omega$ be the subset where the hypersurface satisfies the uniform bound $|\nabla A| \leq C_B$.  Using interior estimates for mean curvature flow, we can choose $C_B$ large enough that any time one flow starting in $B_s$ satisfies this bound and continues to do so even after applying the balancing map $\gamma$.

Note that $g_{n-1} (\PG^n (v))$ is in $\Omega^B$.  To complete the proof, we will show that there exists $\delta > 0$ so that the action of the conformal linear group on $B_{\delta}$ does not intersect $\Omega^B$.  We will argue by contradiction.   Suppose, thus,  that there exist $v_i \in B_{2^{-i}}$, $g_i \in \cR$ and $h_i \in \cG$ with
\begin{align}
 	g_i (h_i (v_i))) \in \Omega^B \, .
\end{align}
Since $v_i \to 0$, we have that $\rho(v_i) \to (0,1)$.  It follows that $h_i \to 0$.  Since $\cR$ is compact, we can pass to a subsequence so that the $g_i$'s converge to some $\bar{g}$.  It follows that  $g_i (h_i (v_i))) \to \bar{g} (0)$, i.e., they converge to something that is not in $\Omega$.  However, $\Omega^B$ is compactly contained in $\Omega$ and, thus, the limit must be in $\Omega$.  This contradiction completes the proof.
\end{proof}

 \appendix
 
\section{The rescaled MCF equation}	\label{a:lemmas}
 
In this appendix, we will prove    Lemma \ref{t:nonlinQ}.
We will need expressions for geometric quantities for a graph $\Sigma_u$ of a function $u$ over a   hypersurface $\Sigma$,
where $\Sigma_u$ is given by
\begin{equation}
	\Sigma_u = \{ x + u(x) \, \nn (x) \, | \, x \in \Sigma \} \, .
\end{equation}
We will assume that $|u|$ is small so  $\Sigma_u$ is contained in a tubular neighborhood of $\Sigma$ where the normal exponential map is invertible.  Let $e_{n+1}$ be the gradient of the (signed) distance function to $\Sigma$, normalized so that $e_{n+1}$ equals $\nn$ on $\Sigma$.
The geometric quantities    are:
\begin{itemize}
\item The relative area element $\nu_u (p) = \sqrt{\det g^u_{ij}(p)}/ \sqrt{\det g_{ij}(p)}$, where $g_{ij}(p)$ is the metric for $\Sigma$ at $p$ and
 $g^u_{ij}(p)$ is the pull-back metric from $\Sigma_u$.
 \item The mean curvature $H_u(p)$ of $\Sigma_u$ at $(p+ u(p) \, \nn(p))$.
 \item The support function $\eta_u (p) = \langle p + u (p) \, \nn (p) , \nn_u \rangle$, where $\nn_u$ is the normal to $\Sigma_u$.
 \item The speed function $w_u (p) = \langle e_{n+1} , \nn_u \rangle^{-1}$ evaluated at the point $p + u (p) \, \nn(p)$.
\end{itemize}

\vskip1mm
The mean curvature and the support function   appear in the rescaled MCF equation.  The speed function enters indirectly when we rewrite the equation in graphical form;   the speed function adjusts for  that the normal direction and vertical directions may not be the same.  The relative area element is used to compute the mean curvature.  See \cite{EH1}, \cite{EH2} for similar quantities for graphs over a plane.

\vskip1mm
The next lemma from \cite{CM3} (lemma $A.3$ there) computes $\nu_u$, $\eta_u$ and $w_u$:

\begin{Lem} \cite{CM3}	\label{l:areau}
There are   functions $w, \nu , \eta$ depending  on $(p , s , y)  \in \Sigma \times \RR \times T_p\Sigma$  that are smooth for $|s|$ sufficiently small and depend smoothly on $\Sigma$ so that:
\begin{itemize}
\item   $w_u (p) = w(p, u (p) , \nabla u (p))$, $\nu_u (p) = \nu (p ,u (p) , \nabla u(p))$ and $\eta_u (p) = \eta (p, u(p) , \nabla u (p))$.
\end{itemize}
The ratio $\frac{w}{\nu}$ depends only on $p$ and $s$.
Finally, the functions $w$, $\nu$, and $\eta$ satisfy:
\begin{itemize}
\item $w(p,s,0) \equiv 1$,  $\partial_s w(p,s,0) = 0$, $\partial_{y_{\alpha}}  w  (p,s,0) = 0$, and $ 
\partial_{y_{\alpha}} \partial_{  y_{\beta} }w (p,0,0) = \delta_{\alpha \beta}$.
\item  $\nu(p,0,0) =1$; the only non-zero first and second order terms are $ \partial_s \nu  (p,0,0) = H(p)$,
$ \partial_{p_j} \partial_s \nu  (p,0,0) = H_j(p)$,
$\partial_s^2 \nu(p,0,0) = H^2 (p) - |A|^2 (p)$, and 
$\partial_{y_{\alpha}} \partial_{  y_{\beta} }  \nu  (p,0,0) = \delta_{\alpha \beta}$.
\item $\eta (p,0, 0) = \langle p , \nn \rangle$, $ \partial_s  \eta  (p,0,0) = 1$, and $ \partial_{  y_{\alpha}} \eta (p,0,0) = - p_{\alpha}$.
\end{itemize}
\end{Lem}

 Using this, corollary $A.30$ in \cite{CM3} computed
  the mean curvature $H_u$:
  
\begin{Cor} \cite{CM3}	\label{c:graphue}
The mean curvature $H_u$ of $\Sigma_u$ is given by
\begin{align}	\label{e:ymc}
	H_u (p) &= \frac{w}{ \nu } \left[ \partial_s \nu       - \dv_{\Sigma} 
	\left(  \partial_{y_{\alpha}} \nu   \right) \right]   	\, , 
\end{align}
where $\nu$ and its derivatives  are all evaluated at $(p,u(p) , \nabla u(p))$.
\end{Cor}

Using this and Lemma \ref{l:areau} gives the well-known (see, e.g., \cite{HP})  formula for the linearization $L^H$ of $H_u$:
\begin{equation}	\label{l:secvar}
	L^H \, u \equiv \frac{d}{dt} \big|_{t=0} \, H_{tu}  = -\Delta \, u - |A|^2 \, u \, .
\end{equation}

\subsection{The rescaled mean curvature flow over a  shrinker}

   Lemma $A.44$ in \cite{CM3} computes the graphical rescaled MCF equation:

\begin{Lem}	\label{l:normpart1}
\cite{CM3} The graphs $\Sigma_{u}$ flow by rescaled MCF   if and only if $u$ satisfies
\begin{align}
	 \partial_t u(p,t)  =  w (p, u(p,t) , \nabla u (p,t)) \,  \left( \frac{1}{2} \, \eta (p, u(p,t) , \nabla u(p,t)) - H_u \right)   \equiv \cM \, u \, .
\end{align}
\end{Lem}

Using this, we can compute the linearization of $\cM$:
 
\begin{Cor}	\label{c:line}
The linearization of $\cM u$ at $u=0$ is given by
\begin{equation}
	\frac{d}{dr} \, \big|_{r=0} \, \cM (r\, u) = \Delta \, u + |A|^2 \, u - \frac{1}{2} \, \langle p , \nabla u \rangle + \frac{1}{2} \, u 
	= L \, u \, , 
\end{equation}
where $L$ is the second variation operator for the $F$ functional from section $4$ of 
\cite{CM1}.
\end{Cor}

\begin{proof}
Computing directly and using that $L^H$ is  the linearization of $H_u$  gives  
\begin{align}
	\frac{d}{dr} \,& \big|_{r=0} \, \cM (r\, u)   = w(p,0,0) \, \left( \frac{1}{2} \, u \, \partial_s \eta (p, 0,0  )+ \frac{1}{2} u_{\alpha} \, \partial_{y_{\alpha}}
	 \eta (p,0,0) - L^H \, u \right)   \\
	 &\qquad + \left( u \, \partial_s w (p, 0,0  )+   u_{\alpha} \, \partial_{y_{\alpha}}
	 w (p,0,0)\right)
	 \left( \frac{1}{2} \, \eta (p, 0,0) - H_{0} \right)  \notag \, .
\end{align}
Since Lemma \ref{l:areau} gives  
$\partial_s w (p, 0,0  ) = \partial_{y_{\alpha}}
	 w (p,0,0) = 0 $ and $w(p,0,0) = 1$, we get 
\begin{align}
	\frac{d}{dr} \, \big|_{r=0} \, \cM (r\, u)   
	=   \frac{1}{2} \, u \, \partial_s \eta (p, 0,0  )+ \frac{1}{2} u_{\alpha} \, \partial_{y_{\alpha}}
	 \eta (p,0,0) + \Delta \, u  + |A|^2 \, u 
	 \, ,
\end{align}
where the last equality used that $L^H \, u = - \Delta u - |A|^2 \, u$ by \eqr{l:secvar}.
Finally, note that Lemma \ref{l:areau} gives 
$\partial_s \eta (p, 0,0  ) = 1$ and $   \partial_{y_{\alpha}}
	 \eta (p,0,0) = - p_{\alpha}$.
\end{proof}

\subsection{Controlling the nonlinearity}
The nonlinearity $\cQ(u)$ is defined by
$
	\cQ (u) = \cM u - L \, u $, where $L$ is the linearization of $\cM$ at $0$.

\begin{Pro}	\label{p:formofQ}
The nonlinearity $\cQ$ can be written as 
\begin{align}
	\cQ(u) = \bar{f}(p,u,\nabla u)   
	+ \dv_{\Sigma} \left( \bar{W} (p,u,\nabla u)  \right) + \langle \nabla \bar{h} , \bar{V} \rangle  \, , 
\end{align}
where $\bar{f}$ and $\bar{h}$ are    smooth functions and  $\bar{W}$ and $\bar{V}$  are smooth vector fields
with:
\begin{enumerate}
\item[(P1)] $\bar{f}(p,0,0)= \partial_s \bar{f} (p,0,0) = \partial_{y_{\alpha}} \bar{f} (p,0,0) = 0$.
\item[(P2)] $\bar{W}(p,0,0) = \partial_s \bar{W} (p,0,0) = \partial_{y_{\alpha}} \bar{W} (p,0,0) = 0$.
\item[(P3)] $\bar{h} (p,0,0) =0$,   $\partial_s \bar{h} (p,0,0) =H(p)$ and $\partial_{y_{\alpha}} \bar{h} (p,0,0) = 0$.
\item[(P4)]  $\bar{V} (p,0,0) = 0$.
\end{enumerate}
\end{Pro}

\vskip2mm
The point of Proposition \ref{p:formofQ} is that $\cQ (u)$ is essentially quadratic in $u$.   Namely, if $r$ is a small parameter and $u$ is a fixed function, then Proposition \ref{p:formofQ} gives
\begin{equation}	\label{e:easyr}
	|\cQ ( r\, u) |	\leq C_u \, r^2 \, , 
\end{equation}
where $C_u$ is a constant depending on $u$ and bounds for the derivatives of $\bar{f}$ and $\bar{W} $.

\begin{proof}[Proof of Proposition \ref{p:formofQ}]
In  this proof,   $w(0)$  denotes $w(p,0,0)$ and $w$  denotes $w(p,u(p) , \nabla u(p))$; we use the same convention for $\eta (0)$, $\nu (0)$ and other functions of $(p,s,y)$.

Using Corollary \ref{c:graphue} and Lemma \ref{l:normpart1}, 
the operator $\cM$ is given by
\begin{align}	 
	\cM \, u  &= w  \, \left( \frac{1}{2} \, \eta   - H_u \right)    =  \frac{1}{2} \, w\,  \eta  - 
\frac{w^2  \partial_s \nu}{\nu} +
 \frac{w^2}{\nu} \, \dv_{\Sigma} \left( \partial_{y_{\alpha}} \nu \right)  \, .
\end{align}
Define 
$
	\bar{W} =   \frac{w^2}{\nu}    \, \partial_{y_{\alpha}} \, \nu - y_{\alpha} $, so that   
\begin{align}
	\cQ \, u   - \dv_{\Sigma} (\bar{W}) &= \frac{w\, \eta}{2} - \frac{w^2 \partial_s \nu}{\nu}
	+ \frac{w^2}{\nu} \, \dv_{\Sigma} (\partial_{y_{\alpha}} \, \nu) - \Delta u +   \langle \frac{p}{2} , \nabla u \rangle  - |A|^2 \, u - \frac{u}{2}   \notag \\
	&\qquad + \dv_{\Sigma} \left( u_{\alpha} - \frac{w^2 }{\nu}    \, \partial_{y_{\alpha}} \, \nu  \right) \\
	&= \frac{w\, \eta}{2} - \frac{w^2 \partial_s \nu}{\nu} - \langle \nabla \frac{w^2 }{\nu} ,  \partial_{y_{\alpha}} \, \nu \rangle + \frac{1}{2} \langle p , \nabla u \rangle  - |A|^2 \, u - \frac{1}{2} \, u 
	\notag \, .
\end{align}
Hence, we define the vector field $\bar{V}$ by
$\bar{V}  = \partial_{y_{\alpha}} \nu$ and 
functions $\bar{h}  = 1 - \frac{w^2}{\nu}$ and{\footnote{We added $1$ in the definition of $\bar{h}$ to make $\bar{h} (0) = 0$.}}
 \begin{align}
 	\bar{f} &= \frac{w\, \eta}{2} - \frac{w^2\, \partial_s \nu}{\nu}  + \frac{1}{2} \langle p , y \rangle  - |A|^2 \, s - \frac{1}{2} \, s \, 	  .
 \end{align}

It remains to check the (P1)--(P4) using the following results from Lemma \ref{l:areau}:
\vskip2mm
\begin{tabular}{c c c c}
 Function of $(p,s,y)$ & Value at $(p,0,0)$ & $\partial_s$ at $(p,0,0)$  & $\partial_{y_{\beta}}$ at $(p,0,0)$ \\
 $w$ & 1 & 0 & 0 \\
  $\nu$ & 1 & $H(p)$ & 0 \\
   $\eta$ & $\langle p , \nn \rangle$ & 1 & $-p_{\beta}$ \\
$ \partial_{y_{\alpha}} \nu$ &0 & 0 & $\delta_{\alpha \beta}$ \\
$  \partial_s \nu $ & $H(p)$ & $H^2(p) - |A|^2(p)$ & 0 \\
\end{tabular}
\vskip2mm

The first claim in (P1)  follows 
 \begin{align}
 	\bar{f} (0) &= \frac{w (0)\eta(0)}{2} -  \frac{w^2(0) \partial_s \nu (0) }{\nu(0)} = \frac{1}{2} \, \langle p , \nn \rangle - H(p) = 0 \, , 
\end{align}
where the last equality is the  shrinker equation.  For the second claim in (P1), we get
\begin{align}
	\partial_s \bar{f}(0) &= \frac{\eta(0) \partial_s w (0)}{2}  + \frac{w (0)\partial_s \eta(0)}{2} 
	- \frac{w^2(0) \partial^2_s \nu(0)}{\nu(0)} + \frac{w^2(0) (\partial_s \nu(0))^2}{\nu^2 (0)} \notag \\
	&\qquad
	- \frac{ 2 w (0) \, \partial_s w (0) \partial_s \nu (0)}{\nu (0)} - |A|^2 (p) - \frac{1}{2}  \\
	&= 0 + \frac{1}{2} - (H^2(p) - |A|^2(p)) + H^2 (p) - 0 - |A|^2 (p)   - \frac{1}{2}  = 0 \, . \notag
\end{align}
 The last claim in (P1) follows from
 \begin{align}
	\partial_{y_{\beta}} \bar{f}(0) &= \frac{\eta(0) \partial_{y_{\beta}} w (0)}{2}   + \frac{w (0)\partial_{y_{\beta}} \eta(0)}{2} 
	- \frac{w^2(0)\partial_{y_{\beta}} \partial_s \nu(0)}{\nu(0)} 
	+ \frac{w^2(0)\partial_s \nu(0) \, \partial_{y_{\beta}} \nu (0)}{\nu^2 (0)} \notag \\
	&\qquad -
	\frac{2\, w(0) \, \partial_{y_{\beta}} w (0) \partial_s \nu (0)}{\nu (0)}
	 + \frac{1}{2} \, p_{\beta}  = 0 - \frac{1}{2}\, p_{\beta} -0  + 0 - 0 + \frac{1}{2} \, p_{\beta}  = 0 \, .  
\end{align}

Next, we turn   to (P2) and  $\bar{W}$.  The first claim is immediate since
$
	\bar{W} (0) = \frac{w^2 (0)}{\nu(0)}    \, \partial_{y_{\alpha}} \, \nu (0) = 0 
$.
The second claim follows from
\begin{align}
	\partial_s \bar{W} (0) &= \frac{2w(0) \, \partial_s  w(0) }{\nu(0)}    \, \partial_{y_{\alpha}} \, \nu(0) - 
	\frac{w^2 (0) \partial_s \nu (0)}{\nu^2(0)}    \, \partial_{y_{\alpha}} \, \nu (0) + 
	\frac{w^2(0) }{\nu(0)}    \, \partial_s \, \partial_{y_{\alpha}} \, \nu (0) = 0  
\end{align}
since each term vanishes.  The last claim follows from 
\begin{align}
	\partial_{y_{\beta}} \bar{W} (0) &= \frac{2 w(0) \, \partial_{y_{\beta}}  w(0) }{\nu(0)}    \, \partial_{y_{\alpha}} \, \nu(0) - 
	\frac{w^2 (0) \partial_{y_{\beta}} \nu (0)}{\nu^2(0)}    \, \partial_{y_{\alpha}} \, \nu (0) + 
	\frac{w^2(0) }{\nu(0)}    \, \partial_{y_{\beta}} \, \partial_{y_{\alpha}} \, \nu (0)  - \delta_{\alpha \beta} \notag \\
	&= 0 - 0 +  \delta_{\alpha \beta} -  \delta_{\alpha \beta} = 0 \, .
\end{align}

The first part of   (P3) follows since $\nu(p,0,0) =w(p,0,0) = 1$.  The second part uses
\begin{align}
	\partial_s \bar{h} (0)  =  \frac{w^2(0) \, \partial_s \nu (0)}{\nu^2(0)} - \frac{2 \, w(0) \partial_s w (0)}{\nu (0)} = H(p) - 0 = H(p)  \, .
\end{align}
The last claim in (P3) follows  from $ \partial_{y_{\alpha}} \nu (0) = \partial_{y_{\alpha}} w (0) = 0$.

Finally, Property (P4) is immediate since  $\bar{V}= \partial_{y_{\alpha}} \nu$ vanishes at $(p,0,0)$.
\end{proof}

 We will also use   the following elementary
calculus lemma:

\begin{Lem}	\label{l:calculus}
Let $\bar{U}$ be a $C^1$ function of $(p,s,y)$.  If $u$ and $v$ are $C^1$ functions on $\Sigma$, then 
\begin{align}
	\left| \bar{U} (p,u(p), \nabla u(p)) - \bar{U} (p,v(p), \nabla v(p)) \right| &\leq
	C_{\bar{U}} \, \left( \left| u(p) - v(p) \right|  +   \left| \nabla u (p) - \nabla v (p) \right| \right) \, ,
\end{align}
where $C_{\bar{U}} = \sup \{ \left| \partial_s \bar{U} \right| + \left| \partial_{y_{\alpha}} \bar{U} \right|
\, \big|   \, |s| + |y| \leq \| u \|_{C^1} + \| v \|_{C^1} \}$.
\end{Lem}

\begin{proof}
Using the fundamental theorem of calculus and the chain rule gives
\begin{align}
	&\bar{U} (p,u , \nabla u ) - \bar{U} (p,v , \nabla v )   = (u -v) \int_0^1 
	 \partial_s \bar{U} (p,t(u-v) + v, t(\nabla u- \nabla v) + \nabla v) \, dt  \\
	 &\qquad \qquad + \left(  \partial_{p_{\alpha}} (u -v) \right)   \int_0^1  
	\partial_{y_{\alpha}} \bar{U} (p,t(u-v) + v, t(\nabla u- \nabla v) + \nabla v) \, dt \notag \, ,
\end{align}
where $u$, $v$, $\nabla u$ and $\nabla v$ are all evaluated at $p$.
\end{proof}

\begin{proof}[Proof of Lemma \ref{t:nonlinQ}]
 To get (Q), use Proposition
 \ref{p:formofQ} to write $\cQ(u) - \cQ(v)$ as
 \begin{align}
	\cQ(u)- \cQ(v) &= \bar{f}(p,u,\nabla u)   - \bar{f}(p,v,\nabla v)    + \dv_{\Sigma} \left( \bar{W} (p,u,\nabla u)  -  \bar{W} (p,v,\nabla v) 
	 \right)  \\
	 &\qquad + \langle \nabla  \bar{h} (p,u , \nabla u) , \bar{V} (p,u,\nabla u) \rangle 
	 -  \langle \nabla  \bar{h} (p,v , \nabla v) , \bar{V} (p,v,\nabla v) \rangle 
	  \, ,  \notag
\end{align}
where $u$, $v$, $\nabla u$ and $\nabla v$ are all evaluated at $p$.
 Define 	$f (p)  =    \bar{f}(p,u,\nabla u)   - \bar{f}(p,v,\nabla v) $ and 
		$W(p)  =  \bar{W} (p,u,\nabla u)  -  \bar{W} (p,v,\nabla v)$ and
 write the remainder  as
\begin{align}
	  \langle \nabla \bar{h} (p,u , \nabla u) , \bar{V} (p,u,\nabla u)   -   \bar{V} (p,v,\nabla v) \rangle  
	  + \langle \nabla \left(   \bar{h} (p,u , \nabla u)  
	 -   \bar{h} (p,v , \nabla v) \right) , \bar{V} (p,v,\nabla v) \rangle  \, . \notag
\end{align}
Finally, define $\bar{h}_u = \bar{h} (p,u,\nabla u)$,
$\bar{V}_v = \bar{V} (p,v, \nabla v)$, $V   = \bar{V} (p,u,\nabla u)   -   \bar{V} (p,v,\nabla v) $ and{\footnote{We subtracted $H(p) \, (u-v)$  to kill off the  non-zero term in the first order Taylor series of $\bar{h}$.}}
\begin{align}
	h  &= 	\bar{h} (p,u , \nabla u)  
	 -   \bar{h} (p,v , \nabla v) - H(p) \, (u-v)	 
			\, .
\end{align}
  It remains to prove the  bounds for the quantities.

To bound $|f|$, use Lemma \ref{l:calculus} and then (P1) from Proposition
 \ref{p:formofQ}  to get
  \begin{align}
  	|f| \leq C_{\bar{f}} \, \left( |u-v| + |\nabla u - \nabla v| \right)  \leq C \, \left( \|u \|_{C^1} + \|v \|_{C^1} \right) \, 
	\left( |u-v| + |\nabla u - \nabla v| \right)   \, ,
  \end{align}
  where the last inequality used that  $\bar{f}$ is $C^2$ and $ \partial_s \bar{f} (p,0,0)= \partial_{y_{\alpha}} \bar{f}(p,0,0) =0$      to bound $C_{\bar{f}}$ by $C \, \left( \|u \|_{C^1} + \|v \|_{C^1} \right) $ for a constant $C$ depending on the second derivatives of $\bar{f}$.   
  The bound on $W$ follows  similarly since $\partial_s \bar{W}(p,0,0) = \partial_{y_{\beta}} \bar{W}(p,0,0) = 0$ by (P2).
   To bound $h$, apply
  Lemma \ref{l:calculus} to   $\bar{h} - H \, s$ 
and use (P3)  to get $ \partial_{y_{\beta}} \left( \bar{h} - H \, s \right) (p,0,0) = 
	  0$ and 
 \begin{align}
 	\partial_s \left( \bar{h} - H \, s \right) (p,0,0) &= 
	\partial_s   \bar{h}  (p,0,0) - H = 0 \,  	 \, .
 \end{align}
  To bound  $V$, use Lemma \ref{l:calculus} to get{\footnote{We do not get the  smallness in the $C^1$ norms of $u$ and $v$ since the   $y$ derivative  of  $\bar{V}$ is not $0$ at $0$.     }}  that
$
	|V|  \leq C_{\bar{V}} \,  \left( |u-v| + |\nabla u - \nabla v| \right)$.
   
  Next, since $\bar{h} (p,0,0) = 0$ and $\bar{V}(p,0,0)$ by (P3) and (P4)
  in Proposition
 \ref{p:formofQ}, we can apply Lemma \ref{l:calculus}  (with one of the functions equal to $0$) to get
    \begin{align}
  	\left| \bar{h}_u \right|  &= \left| \bar{h} (p,u,\nabla u)| 
	- \bar{h} (p,0,0) \right| \leq C_{\bar{h}} \,  \left( |u| + |\nabla u| \right) \, , \\
	\left| \bar{V}_v \right|  &= \left| \bar{V} (p,v,\nabla v)| 
	- \bar{V} (p,0,0) \right| \leq C_{\bar{V}} \,  \left( |v| + | \nabla v| \right) \, .
  \end{align}
  
  To bound  $|\nabla \bar{h}_u|$, use the chain rule to get
    \begin{align}		\label{e:firstdb}
  	\partial_{p_{\alpha}} \bar{h}_u &    
	 = \partial_{p_{\alpha}}  \bar{h}  (p,u, \nabla u)  
	+ \partial_s  \bar{h}  (p,u, \nabla u)  \, u_{\alpha} + \partial_{y_{\beta}}  \bar{h}  (p,u, \nabla u)  \, u_{\alpha \beta} \, . 
  \end{align}
 For the first term, we use that 
  $\bar{h} (p,0,0) = 0$ by (P3) 
  in Proposition
 \ref{p:formofQ} and, thus also $\partial_{p_{\alpha}}  \bar{h}  (p,0,0) =0$, so we can apply Lemma \ref{l:calculus} to 
  $\partial_{p_{\alpha}}  \bar{h}$, $u$ and $0$ to get
 \begin{align}
 	\left| \partial_{p_{\alpha}}  \bar{h}  (p,u, \nabla u)\right| = 
	\left| \partial_{p_{\alpha}}  \bar{h}  (p,u, \nabla u)  - \partial_{p_{\alpha}}  \bar{h}  (p,0, 0)  \right| \leq
	C \, \left( |u| + |\nabla u| \right)  \, .
 \end{align}
 The second term is bounded by $C \, |\nabla u|$.  For the third term, we use that $\partial_{y_{\beta}} \bar{h} (p,0,0) =0$ by 
 (P3) 
  in Proposition
 \ref{p:formofQ}, so Lemma \ref{l:calculus}  gives
 \begin{align}
 	 \left| \partial_{y_{\beta}}  \bar{h}  (p,u, \nabla u) \right| = \left|  \partial_{y_{\beta}}  \bar{h}  (p,u, \nabla u)   
	 -  \partial_{y_{\beta}}  \bar{h}  (p,0, 0) \right| \leq C \, \left( |u| + |\nabla u| \right)  \, . 
 \end{align}
 Putting these three bounds back into \eqr{e:firstdb} gives
 \begin{align}
 	\left| \nabla \bar{h}_u \right| \leq C \, \left( |u| + |\nabla u| \right) \, \left( 1 + \left| \Hess_u \right| \right) \, .
 \end{align}
The bound on $\dv_{\Sigma} \bar{V}_v$ follows similarly from the chain rule.

It now remains only to prove \eqr{e:lastone35} and \eqr{e:lastone35b}.  %%here
Since $W(p)  =  \bar{W} (p,u,\nabla u)  -  \bar{W} (p,v,\nabla v)$, we have
\begin{align}	\label{e:for311}
	|\partial_{p_{\alpha}} W| &\leq   \left|
	 \bar{W}_{p_{\alpha}} (p,u,\nabla u)  -  \bar{W}_{p_{\alpha}} (p,v,\nabla v)
	\right| + \left|
	 \bar{W}_s (p,u,\nabla u)u_{{\alpha}}  -  \bar{W}_s (p,v,\nabla v)v_{{\alpha}}
	\right| \notag \\
 &+ \left|
	 \bar{W}_{y_{\beta}} (p,u,\nabla u) u_{\alpha \beta} -  \bar{W}_{y_{\beta}} (p,v,\nabla v)v_{\alpha \beta} 
	\right|
			\, .
\end{align}
Using Lemma \ref{l:calculus} and, by (P2), $ \bar{W}_{p_{\alpha}s}  (p,0,0)=  \bar{W}_{p_{\alpha}y_{\beta}} (p,0,0)=0$, we get
\begin{align}
	 \left|
	 \bar{W}_{p_{\alpha}} (p,u,\nabla u)  -  \bar{W}_{p_{\alpha}} (p,v,\nabla v)
	\right| \leq C \, \left( \| u \|_{C^1} + \| v \|_{C^1} \right) \, \left( |u-v| + |\nabla u - \nabla v| \right) \, .
\end{align}
To bound the second term in \eqr{e:for311}, we start with the triangle inequality
\begin{align}
	 \left|
	 \bar{W}_s (p,u,\nabla u)u_{{\alpha}}  -  \bar{W}_s (p,v,\nabla v)v_{{\alpha}}
	\right| &\leq  \left|
	 \bar{W}_s (p,u,\nabla u)  -  \bar{W}_s (p,v,\nabla v) 
	\right| \, |u_{\alpha}| \notag \\
	&\quad +  \left|
	  \bar{W}_s (p,v,\nabla v) \right| \, \left| u_{\alpha} - v_{{\alpha}}
	\right|  	\,  .
\end{align}
We use that $\bar{W}_s$ is  locally Lipschitz to get
\begin{align}
	   \left|
	 \bar{W}_s (p,u,\nabla u)  -  \bar{W}_s (p,v,\nabla v) 
	\right| \, |u_{\alpha}|  \leq C \, |\nabla u| \left( |u-v| + |\nabla u - \nabla v| \right)	\,  .
\end{align}
Similarly, using also that $\bar{W}_s (p,0,0)$ by (P2), we get
\begin{align}
	  \left|
	  \bar{W}_s (p,v,\nabla v) \right| \, \left| u_{\alpha} - v_{{\alpha}}
	\right|  \leq C \, \left( |v| + |\nabla v|\right) \, |\nabla u - \nabla v| 	\,  .
\end{align}
 
 For the last term in \eqr{e:for311}, we begin with the triangle inequality
 \begin{align}	 
	  \left|
	 \bar{W}_{y_{\beta}} (p,u,\nabla u) u_{\alpha \beta} -  \bar{W}_{y_{\beta}} (p,v,\nabla v)v_{\alpha \beta} 
	\right| &\leq  \left|
	 \bar{W}_{y_{\beta}} (p,u,\nabla u)  -  \bar{W}_{y_{\beta}} (p,v,\nabla v) 
	\right| \, |u_{\alpha \beta}| \notag \\
	&\quad +  \left|
	  \bar{W}_{y_{\beta}} (p,v,\nabla v) \right| \, \left| u_{\alpha \beta} - v_{{\alpha \beta}}
	\right|  	\,  .
\end{align}
Since $ \bar{W}_{y_{\beta}}$ is locally Lipschitz, we get
 \begin{align}	 \label{e:worstterm}
	 \left|
	 \bar{W}_{y_{\beta}} (p,u,\nabla u)  -  \bar{W}_{y_{\beta}} (p,v,\nabla v) 
	\right| \, |u_{\alpha \beta}|  \leq C \, \left( |u-v| + | \nabla u - \nabla v| \right) \, | u_{\alpha \beta} |	\,  .
\end{align}
Similarly, using also that $\bar{W}_{y_{\beta}} (p,0,0)$ by (P2), we get
 \begin{align}	 
	 \left|
	  \bar{W}_{y_{\beta}} (p,v,\nabla v) \right| \, \left| u_{\alpha \beta} - v_{{\alpha \beta}}
	\right|  \leq C \, 	 \left( |v| + | \nabla v| \right) \, | u_{\alpha \beta} - v_{{\alpha \beta}}
 | \,  .
\end{align}
Substituting these bounds into \eqr{e:for311} gives the desired bound on $|\nabla W|$ in  \eqr{e:lastone35}. 

To prove   \eqr{e:lastone35b}, first use that
$
h   = 	\bar{h} (p,u , \nabla u)  
	 -   \bar{h} (p,v , \nabla v) - H(p) \, (u-v)$ to get	 
	 \begin{align}	\label{e:1LAST}
	 	|\partial_{p_{\alpha}} h| &\leq |\nabla H| \, |u-v| +  |	\bar{h}_{s} (p,u , \nabla u)  u_{\alpha}
	 -   \bar{h}_{s} (p,v , \nabla v) \, v_{\alpha} - H (u_{\alpha} - v_{\alpha}) | \notag \\
	 & +  |	\bar{h}_{p_{\alpha}} (p,u , \nabla u)  
	 -   \bar{h}_{p_{\alpha}} (p,v , \nabla v) | +  |	\bar{h}_{y_{\beta}} (p,u , \nabla u)  u_{\alpha \beta}
	 -   \bar{h}_{y_{\beta}} (p,v , \nabla v) v_{\alpha \beta} | \, .
	 \end{align}
	 The first and third terms on the right in \eqr{e:1LAST} are clearly bounded by $C\, |u-v|$ and, since $h_{p_{\alpha}}$ is locally Lipschitz, $C\, \left( |u-v| + |\nabla u - \nabla v | \right)$.  To bound the last term in \eqr{e:1LAST}, we use the triangle inequality to get
	  \begin{align}
	 	   |	\bar{h}_{y_{\beta}} (p,u , \nabla u)  u_{\alpha \beta}&
	 -   \bar{h}_{y_{\beta}} (p,v , \nabla v) v_{\alpha \beta} | \leq   |	\bar{h}_{y_{\beta}} (p,u , \nabla u)   
	 -   \bar{h}_{y_{\beta}} (p,v , \nabla v)| \, | u_{\alpha \beta} |  \notag \\
	 &\quad +   |	\bar{h}_{y_{\beta}} (p,u , \nabla u)  | \, |u_{\alpha \beta}
	 -    v_{\alpha \beta} | \\
	 &\leq C \, | \Hess_u | \, \left( |u-v| + |\nabla u - \nabla v | \right) + C \, (|u| + |\nabla u|) \left| \Hess_u - \Hess_v \right|
	 \notag
	 \, , 
	 \end{align}
	 where the last inequality uses that $\bar{h}_{y_{\beta}}$ is locally Lipschitz and $\bar{h}_{y_{\beta}} (p,0,0)=0$.  Finally, to bound the second term on the right in
	  \eqr{e:1LAST}, we use the triangle inequality to get
	  \begin{align}	 
	 	  |	\bar{h}_{s} (p,u , \nabla u)  u_{\alpha}
	 -   \bar{h}_{s} (p,v , \nabla v) \, v_{\alpha}  | &\leq   |\bar{h}_{s} (p,u , \nabla u) - \bar{h}_{s} (p,v , \nabla v)  | \, |u_{\alpha}|  \notag \\
	  &\quad + |\bar{h}_{s} (p,u , \nabla u) | \, |\nabla u - \nabla v| \, ,
	 \end{align}
	which we bound similarly. Combining the various bounds gives   \eqr{e:lastone35b}.
 \end{proof}

\end{document}